\documentclass[a4paper,10pt]{scrartcl}
\pdfoutput=1
\usepackage{ifpdf}
\usepackage[english]{babel}
\usepackage[utf8]{inputenc}
\usepackage[T1]{fontenc}
\usepackage{lmodern}
\usepackage{enumerate}
\usepackage{geometry}
\usepackage{amsfonts,amsmath,amssymb,amsopn}
\usepackage{amsthm}
\usepackage{mathtools}
\usepackage{stmaryrd}
\usepackage{nicefrac}
\usepackage{exscale}
\usepackage{empheq}
\usepackage{ mathrsfs }
\usepackage{comment}

\usepackage[numbers,square]{natbib}
\usepackage{url} 
\usepackage[colorlinks=true,pdfpagelabels,unicode]{hyperref}
\hypersetup{linkcolor=blue, urlcolor=violet, citecolor=blue}
\usepackage[dvipsnames,svgnames,table]{xcolor}
\usepackage{orcidlink} 

\numberwithin{equation}{section}
\allowdisplaybreaks 

\theoremstyle{plain}

\newtheoremstyle{theo}
	{3pt} 
	{3pt} 
	{\itshape} 
	{} 
		{\bfseries} 
	{\\} 
	{ } 
	{\thmname{#1}\thmnumber{ #2.}\thmnote{ - #3}} 
\theoremstyle{theo}

\newtheorem{definition}{Definition}[section]
\newtheorem{lemma}[definition]{Lemma}
\newtheorem{theorem}[definition]{Theorem}
\newtheorem{corollary}[definition]{Corollary}
\newtheorem{proposition}[definition]{Proposition}

\newenvironment{bew}{\begin{proof}[\bfseries Proof:]}{\end{proof}}

\newtheoremstyle{remark}
	{3pt} 
	{3pt} 
	{} 
	{} 
		{\bfseries} 
	{} 
	{ } 
	{\thmname{#1}\thmnumber{ #2.}\thmnote{ - #3}} 
\theoremstyle{remark}

\newcommand{\eps}{\varepsilon}

\newcommand{\R}{\mathbb{R}}
\newcommand{\N}{\mathbb{N}}

\newcommand{\f}[2]{\frac{#1}{#2}}
\newcommand{\weps}{w_\eps}
\newcommand{\wepss}{w_{\eps s}}
\newcommand{\wepsss}{w_{\eps ss}}
\newcommand{\wepst}{w_{\eps t}}
\newcommand{\wepsd}{w_{\eps\nu}}
\newcommand{\wepsds}{w_{\eps \nu s}}
\newcommand{\wepsdss}{w_{\eps\nu ss}}
\newcommand{\wepsdt}{w_{\eps\nu t}}
\newcommand{\vp}{\varphi}
\newcommand{\uvp}{\underline{\varphi}}
\newcommand{\ovp}{\overline{\varphi}}
\newcommand{\urho}{\underline{\rho}}
\newcommand{\wew}{\overset{\wedge}{w}}

\newcommand{\tm}{T_{max}} 
\newcommand{\Om}{\Omega} 
\newcommand{\io}{\int\limits_\Omega} 
\newcommand{\bom}{\overline{\Omega}} 
\newcommand{\uw}{\underline{w}} 
\newcommand{\ow}{\overline{w}} 
\newcommand{\uW}{\underline{W}} 
\newcommand{\uc}{\underline c}
\newcommand{\oc}{\overline c}
\newcommand{\ueta}{\underline \eta}
\newcommand{\oeta}{\overline \eta}
\newcommand{\vertiii}[1]{{\left\vert\kern-0.25ex\left\vert\kern-0.25ex\left\vert #1 
    \right\vert\kern-0.25ex\right\vert\kern-0.25ex\right\vert}}

\makeatletter
\def\@fnsymbol#1{\ensuremath{\ifcase#1\or *\or \ddagger\or \#\or
   \mathsection\or \mathparagraph\or \|\or **\or \dagger\dagger
   \or \ddagger\ddagger \else\@ctrerr\fi}}
\makeatother

\author{
Gregor Flüchter\footnote{gmf@mail.upb.de}\ \orcidlink{0009-0008-3208-3687}\\
{\small Institute of Mathematics, Paderborn University,}\\[-5pt]
{\small 33098 Paderborn, Germany}
}
\title{Formation and Behavior of Dirac Singularities in the Parabolic-Elliptic Keller-Segel System in Dimensions \boldmath$n\geq 3$\unboldmath}
\date{}

\begin{document}
\maketitle
\begin{abstract}
\noindent
{\textbf{Abstract.} We consider nonnegative radially symmetric solutions of the parabolic-elliptic Keller-Segel system
\begin{align*}
\left\lbrace
\begin{array}{r@{}l@{\quad}l}
&u_t=\Delta u-\nabla \cdot \big(u\nabla v\big),\\
&0=\Delta v -\mu + u , \\
\end{array}\right. 
\end{align*}
where $\mu$ is the spatial average of $u$, under homogeneous Neumann boundary conditions in a ball in $\mathbb R^n$ for $n\geq 3$.\\
In two dimensions, it is well established that solutions blowing up in finite time converge to a Dirac profile in the vague topology. In contrast, for $n\geq 3$, blow-up solutions with finite existence time do not appear to exhibit such concentration behavior. By generalizing to measure-valued solutions corresponding to accumulated densities of $u$, we extend the analysis beyond the blow-up time. Within this framework, we establish the existence of a minimal solution
\[
u(t)=\theta(t)\delta_0 + \rho(\cdot,t) dx, \qquad t \geq 0,
\]
where $\rho$ is integrable and $\theta$ is increasing and right-continuous. We further construct a class of initial data for which $\theta(t_0)>0$ for some $t_0>0$, thereby establishing the formation of a Dirac mass at the origin. Unlike in the case $n=2$, the singular mass does not jump to a positive level instantaneously; instead, $\theta$ becomes positive continuously. Moreover, $\theta$ is strictly increasing on $[t_0,\infty)$, and the entire mass is asymptotically absorbed at the origin.
}\medskip

{\noindent\textbf{Keywords:} chemotaxis, Keller-Segel system, finite-time blow-up, measure-valued solutions, Dirac mass}

{\noindent\textbf{MSC (2020):} 35B40, 35B44 (primary); 35A01, 35D30, 35K40, 92C17 (secondary)}

\end{abstract}

\newpage
\section{Introduction}\label{sec1:intro}
Chemotaxis, the directed movement of organisms in response to chemical signals, is a fundamental process in biology, governing phenomena such as bacterial aggregation, immune response and cancer metastasis. In 1970, Keller and Segel proposed a seminal model describing this behavior (\cite{Keller1970}), which has since inspired a wide range of mathematical studies (see \cite{Horstmann2003}, \cite{Bellomo2015} and \cite{Bellomo2022} for surveys).\\
We are concerned with the study of the parabolic-elliptic Keller-Segel system
\begin{equation}\label{KS}
    	\left\{ \begin{array}{rcll}
	u_t &=& \Delta u - \nabla \cdot (u\nabla v), 
	\qquad & x\in\Omega, \ t>0, \\[1mm]
	0 &=& \Delta v - \mu + u,
	\qquad & x\in\Omega, \ t>0, \\[1mm]
	& & \hspace*{-10mm}
	\frac{\partial u}{\partial\nu}=\frac{\partial v}{\partial\nu}=0,
	\qquad & x\in\partial\Om, \ t>0, \\[1mm]
	& & \hspace*{-10mm}
	u(x,0)=u_0(x),
	\qquad & x\in\Omega, 
 	\end{array} \right.
\end{equation}
where $\Om=B_R(0) \subset \R^n$ is a ball of radius $R>0$, $n\geq 3$, and $\mu=\frac1{|\Omega|}\int_\Omega u_0$ is the average of initial mass. The initial data shall be assumed to satisfy
\begin{align}\label{init}
0\leq u_0 \in C^0(\bom)~\text{is radially symmetric and satisfies}~\io u_0 = m
\end{align}
for some $m>0$. Note that the normalization $\mu$ can be explicitly computed as $\mu = \f{mn}{\omega_n R^n}$, where $\omega_n:=n|B_1(0)|$ is the surface area of the unit sphere in $\R^n$.\\
The functions $u$ and $v$ represent the densities of a population of organisms and a chemoattractant, respectively. In many relevant biological applications, the chemical diffuses much faster than the population. Therefore, \eqref{KS} arises as a relevant limit case (\cite{Jaeger1992, Nagai1995}) of the fully parabolic variant introduced in \cite{Keller1970}, modeling an instantaneous response of the chemical to changes in population density.\\
The system \eqref{KS} admits classical solutions
  \begin{eqnarray*}
	\left\{ \begin{array}{l}
	u \in C^0(\bom\times [0,\tm)) \cap C^{2,1}(\bom\times (0,\tm)) \qquad \mbox{and} \\[1mm]
	v\in \bigcap_{q>n} L^\infty_{loc}([0,\tm);W^{1,q}(\Omega)) \cap C^{2,0}(\bom\times (0,\tm)),
	\end{array} \right.
  \end{eqnarray*}
up to some maximal $T_{max}\in(0,\infty]$ (\cite{Fujie2015, Biler1998, Cieslak2008, Djie2010}). In one spatial dimension ($n=1$), all solutions of \eqref{KS} exist globally and remain bounded (\cite{Osaki2001}). By contrast, for $n\geq 2$, there are initial data $u_0\in C^0(\bom)$ such that the corresponding solution blows up in finite time (see, for instance, \cite{Winkler2019, Blanchet2006, Jaeger1992, Nagai1995, Biler1995, Winkler2013}). That is, $T_{max}<\infty$, and thus due to an extensibility criterion necessarily
\[
\limsup_{t \nearrow T_{max}} \|u(\cdot,t)\|_{L^\infty(\Om)} = \infty,
\]
which can be interpreted as an extreme form of aggregation. In 1973, Nanjundiah conjectured that for radially symmetric solutions in two dimensions, $u$ concentrates into a Dirac mass at the blow-up time (\cite{Nanjundiah1973}). Subsequent work confirmed this behavior (see, e.g., \cite{Herrero1996, Senba2001, Senba2004, Biler2008, Kavallaris2009}): There exist $m\geq 8\pi$ and $\rho\in L^1(\Om)$ such that
\[
u(\cdot,t) \rightharpoonup m \delta_0 + \rho \qquad \text{in }\mathcal M(\Om) \qquad \text{as }t \nearrow T_{max}.
\]
In contrast, in dimensions $n\geq 3$, classical radially symmetric blow-up solutions with finite existence time $T_{max}$ seem not to develop an asymptotic Dirac mass. As indicated by blow-up profiles of certain self-similar solutions in $\R^n$ (\cite{Herrero1997, Senba2005}), Souplet and Winkler (\cite{Souplet2018}) proved that more generally, for radially decreasing initial data, the blow-up profile $U$, attained as the $L^1(B_R(0))$-limit of $u(\cdot,t)$ as $t\nearrow T_{max}$, is bounded via
\[
U(x)\leq C |x|^{-2}, \qquad 0<|x|\leq R,
\]
for some $C>0$, and is therefore integrable. Furthermore, if the initial datum $u_0\in C^1(\bom)$ fulfills, in radial variables, \[
r^{n-1}u_{0r} + u_0(r) \int_0^r \Big(u_0(s)-\mu \Big) s^{n-1} ds \geq 0 \qquad \text{for all }r\in(0,R),
\]
then a matching lower bound holds (\cite{Souplet2018}): There exists $c>0$ such that
\[
U(x)\geq c |x|^{-2}, \qquad 0<|x|\leq R.
\]
In the whole-space setting, for radially decreasing initial data satisfying a condition met in particular by $u_0\in L^1(\R^n)$, it was moreover shown that the limit
\[
\lim_{|x|\searrow 0} |x|^2 U(x) = \alpha \in [0,\infty)
\]
exists (\cite{Bai2025}).\\
The objective of this paper is to investigate, for $n\geq 3$, whether \eqref{KS} admits the formation of Dirac-type singularities, albeit perhaps after blow-up time, and, when such singularities occur, to analyze their behavior.
To this end, we introduce a generalized notion of solution that accomodates Radon measures, thereby allowing for singular mass concentration within the solution framework.

\begin{definition}
Let $n\geq 1$ and $\Om=B_R(0)\subset \R^n$ for some $R>0$, $m>0$ and $\mu = \f{mn}{\omega_n R^n}$. Furthermore, suppose that $u_0$ satisfies \eqref{init} and that for almost every $t>0$, $(u(t),v(t))$ is a tuple of radially symmetric measures such that
\[
	\left\{ \begin{array}{l}
	u\in L_{loc}^\infty([0,\infty); \mathcal M_+(\bom)) \qquad \mbox{and}\\[1mm]
	v\in L_{loc}^1([0,\infty);W^{1,1}(\Om))\cap C^{1,0}((\bom\setminus\{0\})\times(0,\infty)).
	\end{array} \right.
\]
We call $(u,v)$ a mass-conserving radially symmetric measure-valued solution of \eqref{KS}, if
\begin{align}\label{wdeffrommeasureacc}
w(s,t):=\f{u(t)(B_{s^\f1n}(0))}{\omega_n} \qquad \text{for }s\in(0,R^n]~\text{and }t\geq 0
\end{align}
defines a classical solution $w \in C^0((0,R^n]\times[0,\infty))\cap C^{2,1}((0,R^n]\times(0,\infty))$ of the incomplete Dirichlet problem
\begin{align}\label{0wacc}
\left\lbrace
\begin{array}{r@{}l@{\quad}l@{\quad}l@{\,}c}
	&w_t = n^2 s^{2-\frac{2}{n}} w_{ss} + nww_s - \mu s w_s,
	\qquad & s\in (0,R^n), \ t>0, \\[1mm]
	&w(R^n,t)=\frac{m}{\omega_n},
	\qquad & t>0, \\[1mm]
	& w(s,0)=w_0(s),
	& s\in (0,R^n),
\end{array}\right.
\end{align}
with $w_s(s,t)\geq 0$ for $s\in(0,R^n)$ and $t\geq 0$, where $w_0$ is given by
\begin{align}\label{w0acc}
	w_0(s):=\int_0^{s^\frac{1}{n}} \rho^{n-1} u_0(\sigma) d\sigma,
	\qquad s\in [0,R^n],
\end{align}
and $v$ satisfies
\begin{align}\label{vdistreqacc}
\int_{\bom} \nabla v(\cdot,t) \cdot \nabla \psi dx = \int_{\bom} \psi d(\mu dx - u(t))
\end{align}
for all radially symmetric $\psi \in C^\infty(\bom)$ and $t\in(0,\infty)$.\\
A mass-conserving radially symmetric measure-valued solution $(u,v)$ of \eqref{KS} is called minimal if, for every other such solution $(\tilde u, \tilde v)$,
\begin{align*}
u(t)(B_r(0)) \leq \tilde u(t)(B_r(0)) \qquad \text{for all }r\in(0,R]~\text{and }t\geq 0
\end{align*}
holds.
\end{definition}

\textbf{Remarks.} \begin{itemize}
\item[(i)] For radially symmetric classical solutions, that is $u\in C^0(\bom\times[0,T_{max}))\cap C^{2,1}(\bom\times(0,T_{max}))$ of \eqref{KS} for some $T_{max}>0$, after a change to radial coordinates $u(r,t):=u(x,t)$, where $r=|x|$, \eqref{0wacc} arises via the transformation
\begin{align*}
w(s,t):=\int\limits_0^{s^\frac{1}{n}} \sigma^{n-1}u(\sigma,t)d\sigma,~~s=r^n \in [0,R^n],~t\in[0,T_{max}), 
\end{align*}
including the boundary condition $w(0,t)=0$ for all $t\in [0,T_{max})$ (\cite{Jaeger1992}). If we interpret $u(t)$ as a measure for $t\in[0,T_{max})$, this definition is equivalent to \eqref{wdeffrommeasureacc}.
\item[(ii)] The boundary condition $w(R^n,t)=\frac{m}{\omega_n}$ implies mass conservation, that is $\int_{\bom} du(t) = \omega_n w(R^n,t) = m$ for all $t\geq 0$.
\item[(iii)] A potential Dirac singularity in the origin is given by $u(t)(\{0\})=w(0,t):=\lim_{s\searrow 0} w(s,t)$ for $t\geq 0$, where the limit is guaranteed to exist as $w_s\geq 0$ in $(0,R^n)\times[0,\infty)$.
\item[(iv)] Since \eqref{KS} admits a unique classical solution up to some $T_{max}>0$, the minimal mass-conserving radially symmetric measure-valued solution $(u,v)$ coincides with it in $[0,T_{max})$ if we identify $u(\cdot,t)$ with its density $\rho(\cdot,t)$ for $t\in[0,T_{max})$. Hence, this concept may be regarded as a continuation of the classical solution beyond its potential blow-up time.
\end{itemize}
Our main theorem establishes that, for $n\geq 3$, there exist initial data $u_0\in C^0(\bom)$ leading to the formation  of a Dirac-type mass in the origin. In contrast to $n=2$, this collapse happens slowly in the sense that $\theta$ contains no jump from zero to a positive mass level. Moreover, the concentration of mass is progressing strictly afterward and asymptotically absorbs the entire mass of the system.

\begin{theorem}\label{maintheoremacc}
Let $n\geq 3$ and $\Om=B_R(0)\subset \R^n$ for some $R>0$, $m>0$, $\mu = \f{mn}{\omega_n R^n}$, and additionally assume that $u_0\in C^0(\bom)$ satisfies \eqref{init}. Then there exists a minimal mass-conserving radially symmetric measure-valued solution $(u,v)$ of \eqref{KS} which is such that
\[
u(t)=\theta(t)\delta_0 + \rho(\cdot,t) dx, \qquad t \geq 0,
\]
where $\rho \in C^0((\bom\setminus\{0\})\times[0,\infty)) \cap C^{1,1}((\bom\setminus\{0\})\times(0,\infty))$ and $\theta \in L^\infty((0,\infty))$ are nonnegative functions, and $\theta$ is monotone increasing and right-continuous.\\
Moreover, for $\gamma\in(0,1-\f2n)$, $c>0$, $\tau>0$ and $\eta>0$ there exists $\delta>0$ with the property that if, additionally, $u_0$ is such that $w_0$ as in \eqref{w0acc} fulfills
\begin{align*}
w_0(s) \geq c \cdot \f{s^2}{s^{2-\gamma}+\delta} \qquad \text{for all }s\in[0,R^n],
\end{align*}
then for some $t_0 \in (0,\tau)$, we have that
\begin{align*}
0<\theta(t_0)<\eta.
\end{align*}
Furthermore, $\theta$ is strictly increasing in $[t_0,\infty)$ and satisfies
\begin{align*}
\lim_{t \to \infty} \theta(t) = m.
\end{align*}
\end{theorem}

\subsection{Ideas and results on the level of accumulated densities}\label{subpropos}

We construct solutions $w\in C^0((0,R^n]\times [0,\infty)) \cap C^{2,1}((0,R^n]\times(0,\infty))$ to \eqref{0wacc}, where the initial data $w_0$ satisfies
\begin{align}\label{initacc}
w_0 \in W^{1,\infty}_{loc}((0,R^n]), \quad w_{0s} \geq 0, \quad w_0(0)=0, \quad w_0(R^n)=\f{m}{\omega_n},
\end{align}
as monotone increasing limits of solutions of suitable linear systems. These linear systems are themselves approximated by non-degenerate systems (Lemma~\ref{wepsdlem}--\ref{createdsol}). This also bestows minimality upon these solutions (Lemma \ref{minimalsolacc}).
A key ingredient, not only in this section but throughout the entire paper, is a comparison principle which allows for discontinuities at spatial zero for the supersolution (Lemma \ref{halfcomparison}).\\
By differentiating the singularly weighted functional
\[
\int_0^{R^n} s^{-1+\f2n} w^q ds
\]
for $q\in(0,1)$ with suitable approximations, we shall establish that for any $t_0\geq 0$ and $\tau>0$,
\begin{align*}
\int\limits_{t_0}^{t_0+\tau} \int\limits_0^{R^n} s^{-1+\f2n} w^q w_s ds dt \leq C(1+\tau)
\end{align*}
(Lemma \ref{magicprep}), which we subsequently employ in Lemma \ref{magic} to obtain the crucial estimate
\begin{align}\label{magicintroacc}
\int\limits_{t_0}^{t_0+\tau} s^{-\gamma}w dt \leq C(\gamma)
\end{align}
for all $s\in(0,R^n)$ and $\gamma\in(0,1-\f2n)$ if 
\begin{align}\label{w0at0introacc}
w(0,t)=0 \qquad \text{for all }t\in(t_0,t_0+\tau).
\end{align}
In turn, we construct a subsolution of the form $\uw(s,t):=y(t)\cdot \vp(s)$, where $y>0$ and $\vp(s)$ is chosen sufficiently close to $s^\gamma$ for some $\gamma\in\left(0,1-\f2n\right)$, in such a way as to violate \eqref{magicintroacc} and thereby contradict \eqref{w0at0introacc} (Lemma \ref{compfctmassacc}--\ref{estwg0acc}). At least for the minimal solution, though, there is no immediate jump to a positive mass level, as a comparison argument involving the approximating solutions in Lemma \ref{slowcollapseacc} shows.\\
In Subsection \ref{subsectionmonacc}, Lemma \ref{halfcomparison} cannot be applied directly to conclude that $w(0,\cdot)$ is increasing on $[0,\infty)$. Instead, we show that for any $t_0\geq 0$ and $s_0 \in (0,R^n)$,
\[
w \geq \inf_{s\in(0,R^n)} w(s,t_0) \qquad \text{in }[s_0,R^n]\times[t_0,\infty).
\]
This inequality implies the desired monotonicity upon letting $s_0 \searrow 0$.
To that end, we construct a family of subsolutions that exhibit the key property of being increasing in time, enabling us to control them from below by their initial data (Lemma \ref{uWstartingdata}).\\
With the monotonicity of $w(0,\cdot)$ established, for any $\tau_0>0$ we can then design a subsolution $\uw$ in $[0,s_0]\times[t_0,t_0+\tau_0)$ for some $s_0\in(0,R^n)$ such that 
\[
\lim_{t\nearrow t_0+\tau_0} \uw(s,t) \geq \inf_{s\in(0,R^n)}w(s,t_0) + \eta \qquad \text{for all }s\in(0,s_0]
\]
for some $\eta>0$ (Lemma \ref{uvpacclemma}--\ref{strictmonotonacclemma}). Consequently, $w(0,\cdot)$ is even strictly increasing.\\
A similar argument, incorporating a suitable supersolution for the approximating solutions, further shows that $w(0,\cdot)$ is right-continuous when $w$ is the minimal solution of \eqref{0wacc} (Lemma \ref{ovpacclemma}--\ref{rightcontinuityacc}).\\
Moreover, the only nondecreasing positive steady state of \eqref{0wacc} is $W \equiv \f{m}{\omega_n}$ (Lemma \ref{staticsollemma}).
Consequently, since the time-increasing subsolutions constructed in Lemma \ref{risingsubsolsol} are themselves solutions of \eqref{0wacc} corresponding to smaller initial data, we deduce that $w(\cdot,t)\to W\equiv \f{m}{\omega_n}$ in $C^0_{loc}((0,R^n])$ whenever $w(0,t)$ becomes positive (Lemma \ref{wtomomeganforsgreater0acclemma}).
Thereupon, an iterative application of subsolutions, again relying on Lemma \ref{functionpsidiffineqpluginfct}, yields the convergence $w(0,t) \to \f{m}{\omega_n}$ as $t\to\infty$.\\
Hence, we obtain the following results on the existence and behavior of solutions to \eqref{0wacc}:
\begin{proposition}\label{prop1acc}
Let $n\geq 3$, $R>0$, $m>0$, $\mu = \f{mn}{\omega_n R^n}$, and suppose that $w_0$ complies with \eqref{initacc}.\\
Then there exists a minimal nonnegative classical solution $w\in C^0((0,R^n]\times [0,\infty)) \cap C^{2,1}((0,R^n]\times(0,\infty))$ of \eqref{0wacc} which satisfies $w_s>0$ in $(0,R^n)\times(0,\infty)$; in particular, $w(0,t):=\lim_{s \searrow 0} w(s,t)$ is well-defined for all $t>0$ and defines a right-continuous and increasing function.\\
Moreover, for any choice of $\gamma_i \in(0,1-\f2n)$ for $i\in\{1,2\}$, $\uc > 0,~\oc > 0$, $\tau>0$ and $\eta>0$, there exists $\delta>0$ such that if $w_0$ fulfills the inequalities
\begin{align}\label{prop1accw0bounds}
\uc \cdot \f{s^2}{s^{2-\gamma_1}+\delta} \leq w_0(s) \leq \oc \cdot s^{\gamma_2} \qquad \text{for all }s\in[0,R^n],
\end{align}
then $w$ has the property that
\begin{align*}
0<w(0,t_0)<\eta \qquad \text{for some }t_0\in (0,\tau).
\end{align*}

\end{proposition}

\begin{proposition}\label{prop2acc}
Let $n\geq 3$, $R>0$, $m>0$, $\mu = \f{mn}{\omega_n R^n}$, and assume $w_0$ adheres to \eqref{initacc}. Furthermore, suppose that $w\in C^0((0,R^n]\times [0,\infty)) \cap C^{2,1}((0,R^n]\times(0,\infty))$ is a nonnegative classical solution of \eqref{0wacc} satisfying $w_s\geq 0$ in $(0,R^n)\times(0,\infty)$ which is such that
\begin{align*}
\inf_{s\in(0,R^n)} w(s,t_0) > 0 \qquad \text{for some }t_0\geq 0.
\end{align*}
Then $w(0,t):=\inf_{s\in(0,R^n)} w(s,t)$ is strictly increasing in $[t_0,\infty)$, and
\begin{align*}
w(0,t) \to \f{m}{\omega_n} \qquad \text{as }t\to\infty.
\end{align*}

\end{proposition}

\section{Comparison principle and solution construction}

The foundation of our arguments consists in a comparison principle for \eqref{0wacc}. In line with our objective that $w(0,t_0)>0$ for some $t_0\geq 0$ however, we should neither assume that $w_s$ is bounded nor that $w$ is continuous in $s=0$.
We handle the former by only imposing boundedness on the spatial derivative of the subsolution. For the latter, to compensate for the lack of regularity and therefore missing boundary information on the left side, we utilize a lower bound, that is mostly nonnegativity, as a quasi-boundary condition.
We also refrain from requiring the sub- and supersolution to be twice differentiable at every point in $(0,R^n)$ for fixed $t>0$ (confer \cite[Lemma 5.1]{Bellomo2017}); this relaxation streamlines the constructions in Lemma \ref{staticsubsolacc} and Lemma \ref{functionpsidiffineqpluginfct}.
We thus formulate

\begin{lemma}\label{halfcomparison}
Let $\mu\in\R$, $R>0$, $t_0 \geq 0,~T\in(t_0,\infty]$ and $a\in\R$, and suppose that
\begin{align}\label{comp1}
\uw \in C^0([0,R^n]\times[t_0,T)) \cap L^\infty_{loc}([t_0,T); W^{1,\infty}((0,R^n))) \quad \text{and} \quad
\ow \in C^0((0,R^n]\times[t_0,T)) 
\end{align}
are functions such that
\begin{align}\label{uwowgeqaacc}
\ow(s,t) \geq a \qquad \text{for all }(s,t)\in(0,R^n]\times[t_0,T)
\end{align}
and
\begin{align}\label{comp2}
\uw_t, \uw_s, \ow_t, \ow_s \quad \text{belong to } C^0((0,R^n)\times(t_0,T)),
\end{align}
and that for each $t\in(t_0,T)$ there exists a discrete set $N(t) \subset (0,R^n)$ with the property that
\begin{align}\label{compregt}
\uw(\cdot,t) \in C^2((0,R^n)\setminus N(t)) \quad \text{and} \quad \ow(\cdot,t) \in C^2((0,R^n)\setminus N(t)),
\end{align}
and that
\begin{align}\label{compdiffinequ}
\uw_t(s,t) \leq n^2 s^{2-\f2n} \uw_{ss}(s,t) + n \uw(s,t)\cdot f(\uw_s(s,t)) - \mu s \uw_s(s,t) \qquad \text{for all } t \in (t_0,T)~\text{and}~s \in (0,R^n)\setminus N(t)
\end{align}
as well as
\begin{align}\label{compdiffineqo}
\ow_t(s,t) \geq n^2 s^{2-\f2n} \ow_{ss}(s,t) + n \ow(s,t)\cdot f(\ow_s(s,t)) - \mu s \ow_s(s,t) \qquad \text{for all } t \in (t_0,T)~\text{and}~s \in (0,R^n)\setminus N(t)
\end{align}
for some $f \in W^{1,\infty}_{loc}(\R)$.
If moreover
\begin{align}\label{comp0}
\uw(0,t) \leq a \qquad \text{for all } t \in[t_0,T)
\end{align}
and
\begin{align}\label{compRn}
\uw(R^n,t) \leq \ow(R^n,t)  \qquad \text{for all } t \in[t_0,T),
\end{align}
as well as
\begin{align}\label{compinit}
\uw(s,t_0)\leq \ow(s,t_0) \qquad \text{for all } s \in(0,R^n],
\end{align}
then
\begin{align}\label{comporder}
\uw(s,t)\leq \ow(s,t) \qquad \text{for all } s \in(0,R^n]~\text{and}~t\in[t_0,T).
\end{align}
\end{lemma}

\begin{bew}
In order to prepare our contradictory argument, we fix $T_0 \in(t_0,T)$ and then obtain from (\ref{comp1}) and (\ref{comp2}) that $\uw_s$ is bounded and continuous on $(0,R^n)\times (t_0,T_0)$, so that it is possible to fix some large $\kappa>0$ such that
\begin{align}\label{compkappa}
\kappa \geq 2n |f(\uw_s(s,t))| \qquad \text{for all}~s\in(0,R^n)~\text{and}~t\in(t_0,T_0).
\end{align}
We next fix $\eta>0$ and may then use (\ref{comp0}) along with the uniform continuity of $\uw$ in $[0,R^n]\times[t_0,T_0]$ to find some suitably small $s_*\in(0,R^n)$ such that
\begin{align}\label{competa}
\uw(s,t) \leq a+\f\eta 2 \qquad \text{for all}~(s,t)\in[0,s_*]\times[t_0,T_0].
\end{align}
Thereupon, we can introduce
\[
z(s,t):= \ow(s,t) - \uw(s,t) + \eta e^{\kappa(t-t_0)}, \qquad s\in[s_*,R^n],~t\in[t_0,T_0].
\]
Then $z$ belongs to $C^0([s_*,R^n]\times [t_0,T_0])$ and satisfies
\[
z(s,t_0) \geq  \eta > 0 \qquad \text{for all } s\in[s_*,R^n],
\]
due to (\ref{compinit}), where in particular
\[
\bar t := \sup \{t\in(t_0,T_0)~\vert~z>0~\text{in}~[s_*,R^n]\times[t_0,t] \}
\]
is well-defined. Now if we had $\bar t<T_0$, then there would exist $\bar s \in [s_*,R^n]$ such that
\begin{align}\label{comp0cross}
z(\bar s, \bar t) = 0,
\end{align}
where since
\[
z(s_*,t)\geq a-\uw(s_*,t)+\eta e^{\kappa(t-t_0)} \geq -\f\eta 2 + \eta > 0 \qquad \text{for all}~t\in[t_0,T_0]
\]
by \eqref{uwowgeqaacc} and (\ref{competa}), and since
\[
z(R^n,t) \geq \eta e^{\kappa(t-t_0)} \geq \eta > 0 \qquad \text{for all}~t\in[t_0,T_0]
\]
thanks to (\ref{compRn}), we have that actually $\bar s \in (s_*,R^n)$ and hence clearly
\begin{align}\label{compscross}
z_s(\bar s, \bar t) = 0,
\end{align}
that is $\ow_s(s,t)=\uw_s(s,t)$, and
\begin{align}\label{comptcross}
z_t(\bar s, \bar t)\leq 0
\end{align}
by definition of $\bar t$.\\
Now in order to make appropriate use of (\ref{compregt}), (\ref{compdiffinequ}) and (\ref{compdiffineqo}), we note that both in the generic case when $\bar s \notin N(\bar t)$ as well as in the exceptional situation when $\bar s \in N(\bar t)$, by discreteness of $N(\bar t)$ we can find $(s_j)_{j\in\N} \subset (s_*, \bar s)\setminus N(\bar t)$ such that $s_j \nearrow \bar s$ as $j \to \infty$ and
\begin{align}\label{compzss}
\limsup_{j\to\infty} z_{ss}(s_j,\bar t) \geq 0,
\end{align}
for otherwise there would exist $\hat s \in (s_*, \bar s)$ and $c_1>0$ such that $(\hat s, \bar s) \cap N(\bar t) = \emptyset$ and
\[
z_{ss}(s,\bar t) \leq -c_1 \qquad \text{for all } s \in(\hat s, \bar s),
\]
together with (\ref{comp0cross}) and (\ref{compscross}) leading to the absurd conclusion that
\[
z(s, \bar t) \leq -\f{c_1}2(\bar s - s)^2 < 0 \qquad \text{for all } s\in(\hat s, \bar s)
\]
(confer \cite[Lemma 5.1]{Bellomo2017}). Thus having (\ref{compzss}), we may proceed in a standard manner by firstly evaluating (\ref{compdiffinequ}) and (\ref{compdiffineqo}) at $s=s_j$ and $t = \bar t$ to see that
\begin{align*}
z_t(s_j,\bar t) \geq& n^2 s_j^{2-\f2n} z_{ss}(s_j,\bar t) \\
&+ n \ow(s_j, \bar t) \Big\{ f(\ow_s(s_j, \bar t))-f(\uw_s(s_j, \bar t)) \Big\}+ n z(s_j, \bar t) f(\uw_s(s_j, \bar t)) \\
&-\mu s_j z_s(s_j, \bar t) + \Big \{ \kappa - n f(\uw_s(s_j, \bar t)) \Big\}\cdot \eta e^{\kappa (\bar t-t_0)}
\end{align*}
for all $j \in \N$, and by then letting $j \to \infty$ to infer that (\ref{compzss}), (\ref{compscross}), (\ref{comptcross}), (\ref{comp0cross}) and the continuity of $z_t(\cdot,\bar t),~z_s(\cdot, \bar t)$ and $z(\cdot, \bar t)$ allow for the conclusion that
\[
0 \geq z_t(\bar s, \bar t) \geq \Big\{\kappa - n f(\uw_s(\bar s, \bar t)) \Big\} \cdot \eta e^{\kappa (\bar t-t_0)} > 0
\]
due to (\ref{compkappa}).\\
This contradiction shows that indeed $\bar t=T_0$ and that thus
\[
\ow(s,t) \geq \uw(s,t) - \eta e^{\kappa(t-t_0)} \qquad \text{for all } s\in [s_*,R^n]~\text{and}~t\in[t_0,T_0].
\]
Together with (\ref{competa}), again thanks to \eqref{uwowgeqaacc}, this entails that in fact
\[
\ow(s,t) \geq \uw(s,t) - \eta e^{\kappa(T_0-t_0)} \qquad \text{for all } s\in(0,R^n),~t\in[t_0,T_0),
\]
so that since our choice of $\kappa$ did not depend on $\eta$, we may firstly let $\eta \searrow 0$ and then $T_0 \nearrow T$ to end up with (\ref{comporder}).
\end{bew}

In order to construct solutions to (\ref{0wacc}) in a manner such that $w_s$ is nonnegative, as an approximation we consider

\begin{align}\label{0wepsacc}
\left\lbrace
\begin{array}{r@{}l@{\quad}l@{\quad}l@{\,}c}
	&\wepst = n^2 s^{2-\frac{2}{n}} \wepsss + n\weps f_\eps(\wepss) - \mu s \wepss,
	\qquad & s\in (0,R^n), \ t>0, \\[1mm]
	&\weps(0,t)=0, \quad \weps(R^n,t)=\frac{m}{\omega_n},
	\qquad & t>0, \\[1mm]
	& \weps(s,0)=w_{0\eps}(s):= \min \{\f s\eps, w_0(s) \}, \qquad
	& s\in (0,R^n),
\end{array}\right.
\end{align}
for $\eps \in (0,\f{R^n \omega_n}{m})$, where $(f_\eps)_{\eps \in (0,\f{R^n \omega_n}{m})} \subset C^\infty([0,\infty))$ is increasing with respect to $\eps$ and satisfies
\begin{align}\label{fepsacc}
\left\lbrace
\begin{array}{r@{}l@{\quad}l@{\quad}l@{\,}c}
	&f_\eps(\xi) = \xi, \quad \text{for all } \xi < \f1\eps,\\[1mm]
	&f_\eps \leq \f1\eps +1, \\[1mm]
	&f_\eps' \geq 0.
\end{array}\right.
\end{align}
To that end, we first regard yet another class of approximating systems which are non-degenerate.

\begin{lemma}\label{wepsdlem}
Let $n\geq 1$, $R>0$, $m>0$, $\mu := \f{mn}{\omega_n R^n}$, $\eps \in (0,\f{R^n \omega_n}{m})$ and suppose $w_0$ satisfies (\ref{initacc}).\\
For $\nu \in (0,1)$, consider
\begin{align}\label{0wepsdacc}
\left\lbrace
\begin{array}{r@{}l@{\quad}l@{\quad}l@{\,}c}
	&\wepsdt = (\nu+n^2 s^{2-\frac{2}{n}}) \wepsdss + n\wepsd f_\eps(\wepsds) - \mu s \wepsds,
	\qquad & s\in (0,R^n), \ t>0, \\[1mm]
	&\wepsd(0,t)=0, \quad \wepsd(R^n,t)=\frac{m}{\omega_n},
	\qquad & t>0, \\[1mm]
	& \wepsd(s,0)=w_{0\eps\nu}(s), \qquad
	& s\in (0,R^n),
\end{array}\right.
\end{align}
with $f_\eps \in C^\infty([0,\infty))$ as in (\ref{fepsacc}), and where $(w_{0\eps\nu})_{\nu\in(0,1)}\subset C^2([0,R^n])$ is such that
\begin{align*}
w_{0\eps\nu s} \geq 0, \quad w_{0\eps\nu}(0)=0, \quad w_{0\eps\nu}(R^n)=\f m{\omega_n}, \quad \|w_{0\eps\nu}\|_{W^{1,\infty}((0,R^n))} \leq c_0(\eps)
\end{align*}
for some $c_0(\eps)>0$ and for all $\nu\in(0,1)$, and
\begin{align}\label{w0epsdconv}
w_{0\eps\nu} \to w_{0\eps} \qquad \text{in}~L^\infty((0,R^n)) \quad \text{as}~\nu \searrow 0.
\end{align}
Then (\ref{0wepsacc}) has a unique global classical solution $\wepsd \in C^0([0,R^n]\times[0,\infty)) \cap C^\infty((0,R^n]\times(0,\infty))$ which has the property that $\wepsds \geq 0$ in $(0,R^n)\times(0,\infty)$.
\end{lemma}

\begin{bew}
Since $\nu+n^2 s^{2-\frac{2}{n}}\geq \nu > 0$ for all $s\in(0,R^n)$, and since the source term in (\ref{0wepsdacc}) grows only linearly with respect to $(\wepsd,\wepsds)$ in the sense that
\[
|n \xi_1 f_\eps(\xi_2)- \mu s \xi_2| \leq n\Big(1+\f1\eps\Big)|\xi_1| + \mu R^n |\xi_2| \qquad \text{for all }(\xi_1,\xi_2)\in\R^2~\text{and}~s\in(0,R^n)
\]
according to our construction of $f_\eps$, standard parabolic theory (\cite{Ladyzhenskaja1968}) warrants the existence of a global classical solution $\wepsd$ of (\ref{0wepsdacc}) belonging to the indicated regularity class.\\
As $0\leq w_{0\eps\nu}\leq \f m{\omega_n}$ in $(0,R^n)$, two applications of a standard comparison principle to the constant functions mapping $[0,R^n]\times[0,\infty)$ to $0$ and $\f m{\omega_n}$, respectively, show that $0\leq \wepsd \leq \f m{\omega_n}$ in $[0,R^n]\times[0,\infty)$, which in particular entails that $\wepsds(0,t)$ and $\wepsds(R^n,t)$ are both nonnegative for all $t>0$. Now in view of known results on gradient regularity in scalar parabolic equations (\cite{Ladyzhenskaja1968}), the assumed inclusion $w_{0\eps\nu}\in C^2([0,R^n])$ warrants that $z:=w_{\eps\nu s}$ actually lies in $C^0([0,R^n]\times[0,\infty)) \cap C^{2,1}((0,R^n)\times(0,\infty))$.\\
Therefore, upon observing that
\[
z_t = a_1(s,t) z_{ss} + a_2(s,t) z_s + a_3(s,t) z \qquad \text{in}~(0,R^n)\times(0,\infty)
\]
with
\begin{align*}
&a_1(s,t):= \nu + n^2 s^{2-\f2n},\\
&a_2(s,t):= 2n(n-1)s^{1-\f2n} + n \wepsd(s,t) f_\eps'(\wepsds(s,t))-\mu s \qquad \text{and}\\
&a_3(s,t):= n f_\eps(\wepsds(s,t))-\mu s
\end{align*}
for $(s,t)\in (0,R^n)\times(0,\infty)$, we may once more employ a comparison argument to conclude from the nonnegativity of $w_{0\eps\nu s}$ that indeed also $\wepsds \geq 0$ in $(0,R^n)\times(0,\infty)$.
\end{bew}

By means of a limit process utilizing standard compactness arguments, we can now construct solutions of (\ref{0wepsacc}).

\begin{lemma}\label{wepsolacc}
Suppose that the requirements of Lemma \ref{wepsdlem} are met and $\eps \in (0,\f{R^n\omega_n}m)$. Then there exists $(\nu_j)_{j\in\N}$ such that $\nu_j \searrow 0$ for $j\to\infty$, and 
\[
\wepsd \to \weps~\text{in}~C^0([0,R^n]\times[0,\infty)) \cap C^{2,1}_{loc}((0,R^n]\times(0,\infty)) \qquad \text{as }\nu=\nu_j \searrow 0,
\]
where $\weps \in C^0([0,R^n]\times[0,\infty)) \cap C^{2,1}((0,R^n]\times(0,\infty))$ is a classical solution of (\ref{0wepsacc}) fulfilling $\wepss \geq 0$ in $(0,R^n)\times(0,\infty)$.
\end{lemma}

\begin{bew}
Let $s_0\in(0,R^n)$ and $T>0$. Using that $0\leq \wepsd \leq \f m{\omega_n}$ in $(0,R^n)\times(0,\infty)$ by Lemma \ref{wepsdlem}, and that $0\leq f_\eps \leq 1+ \f1\eps$ on $[0,\infty)$, on applying standard parabolic Schauder estimates (\cite{Ladyzhenskaja1968}) to (\ref{0wepsdacc}), we obtain that for each fixed $\eps \in (0,\f{R^n\omega_n}m)$,
\begin{align}\label{wepsdinreg}
(\wepsd)_{\nu\in(0,1)}~ \text{is bounded in}~C^{2+\theta,1+\f\theta 2}_{loc}([s_0,R^n]\times(0,T])
\end{align}
for some $\theta=\theta(s_0)\in(0,1)$. Moreover, by our assumption that we have
\begin{align}\label{w0epsdW1}
c_1(\eps):=\sup_{\nu\in(0,1)} \|w_{0\eps\nu s}\|_{L^\infty((0,R^n))} < \infty \qquad \text{for all } \eps \in \Big(0,\f{R^n\omega_n}m\Big),
\end{align}
possibly after diminishing $\theta$ we can achieve that also
\begin{align}\label{wepsdoutreg}
(\wepsd)_{\nu\in(0,1)}~\text{is bounded in } C^{\theta,\f\theta 2}([s_0,R^n]\times[0,T]).
\end{align}
In order to derive some regularity information near $s=0$, given $\eps \in (0,\f{R^n\omega_n}m)$, we let $\kappa(\eps):=n(1+\f1\eps)$ and
\[
\ow_\eps(s,t):= c_1(\eps)\cdot e^{\kappa(\eps)t}\cdot s, \qquad s\in[0,R^n],~t\geq 0,
\]
and observe that then according to (\ref{w0epsdW1}),
\begin{align*}
w_{0\eps\nu}(s)&=w_{0\eps\nu}(0)+ \int\limits_0^s w_{0\eps\nu s}(\sigma)d\sigma \\
&\leq c_1(\eps) \cdot s\\
&=\ow_\eps (s,0) \qquad \text{for all } s\in[0,R^n],
\end{align*}
whence clearly also $\wepsd \leq \ow_\eps$ throughout $\{0,R^n\}\times[0,\infty)$ by positivity of $\kappa(\eps)$. Since
\begin{eqnarray*}
&&\ow_{\eps t} - (\nu + n^2 s^{2-\f2n})\ow_{\eps ss} - n \ow_\eps f_\eps(\ow_{\eps s})+\mu s \ow_{\eps s}\\
&=&c_1(\eps) \kappa(\eps)e^{\kappa(\eps)t}\cdot s - n c_1(\eps)e^{\kappa(\eps)t}\cdot s \cdot f_\eps(\ow_{\eps s}) + \mu s \ow_{\eps s}\\
& \geq& c_1(\eps) \kappa(\eps)e^{\kappa(\eps)t}\cdot s - n c_1(\eps)e^{\kappa(\eps)t}\cdot s \cdot \Big(1+\f1\eps\Big)\\
&=& 0 \qquad \text{in}~(0,R^n)\times(0,\infty)
\end{eqnarray*}
due to the inequalities $f_\eps \leq 1+\f1\eps$ and $\ow_{\eps s} \geq 0$, by comparison we conclude that $\wepsd \leq \ow_\eps$ in $[0,R^n]\times[0,\infty)$ for all $\nu\in(0,1)$. Hence, we obtain equicontinuity of $(\wepsd)_{\nu\in(0,1)}$ in $s$ at $s=0$ uniformly for $t\in[0,T]$ for each $\eps \in \Big(0,\f{R^n\omega_n}m\Big)$, as
\[
|\wepsd(s,t)-\wepsd(0,t)|=|\wepsd(s,t)|\leq c_1(\eps)e^{\kappa(\eps)T}\cdot s \qquad \text{for all }s\in[0,R^n],~t\in[0,T],~\nu\in(0,1),
\]
which even yields a Lipschitz bound. Together with (\ref{wepsdoutreg}) and the Arzel\`{a}-Ascoli theorem, this warrants that
\[
(\wepsd)_{\nu\in(0,1)}~\text{is relatively compact in}~C^0_{loc}([0,R^n]\times[0,\infty)).
\]
Along with (\ref{wepsdinreg}) and a further application of the Arzel\`{a}-Ascoli theorem, by means of a straightforward extraction procedure this readily yields $(\nu_j)_{j\in\N}\subset(0,1)$ with $\nu_j \searrow 0$ as $j\to\infty$ as well as a limit function $\weps=\lim\limits_{\nu=\nu_j \searrow 0} \wepsd$ approximated in the claimed topology and, as a consequence thereof and of (\ref{0wepsdacc}) and (\ref{w0epsdconv}), solving (\ref{0wepsacc}) classically and satisfying $\wepss\geq 0$.
\end{bew}

These solutions $w_\eps$ have an ordering property with respect to $\eps$. Therefore, the solution we construct for \eqref{0wacc} is obtained as the monotone limit of $w_\eps$ as $\eps \searrow 0$.

\begin{lemma}\label{createdsol}
Let $n\geq 1$, $R>0$, $m>0$, $\mu := \f{mn}{\omega_n R^n}$ and suppose $w_0$ satisfies (\ref{initacc}).\\
Then there exists a nonnegative function
\begin{align}\label{limitregacc}
w\in C^0((0,R^n]\times [0,\infty)) \cap C^{2,1}((0,R^n]\times(0,\infty))
\end{align}
such that
\begin{align}\label{pointwisewepsacc}
w_\eps \nearrow w \qquad \text{in}~(0,R^n]\times[0,\infty)
\end{align}
as well as
\begin{align}\label{wepswlimitacc}
w_\eps \to w \qquad \text{in}~C_{loc}^0((0,R^n]\times[0,\infty))\cap C_{loc}^{2,1}((0,R^n]\times(0,\infty))
\end{align}
as $\eps \searrow 0$. This limit function is a classical solution of (\ref{0wacc}) in $(0,R^n)\times(0,\infty)$ and satisfies 
\begin{align}\label{wsgeq0asaprep}
w_s \geq 0 \qquad \text{in }(0,R^n]\times (0,\infty);
\end{align}
in particular, $w(0,t):= \lim_{s \searrow 0} w(s,t)$ is well-defined for all $t>0$.
\end{lemma}

\begin{bew}
For $\eps',\eps \in (0,1)$ with $\eps' < \eps < \f{R^n\omega_n}m$, \eqref{fepsacc} implies that $f_\eps \leq f_{\eps'}$ and therefore $\weps$ as in Lemma \ref{wepsolacc} fulfills
\[
\wepst = n^2 s^{2-\frac{2}{n}} \wepsss + n\weps f_\eps(\wepss) - \mu s \wepss \leq n^2 s^{2-\frac{2}{n}} \wepsss + n\weps f_{\eps'}(\wepss) - \mu s \wepss
\]
for all $s\in(0,R^n)$ and $t>0$. Together with the regularity of $\weps$ established in Lemma \ref{wepsolacc} and the boundary information, Lemma \ref{halfcomparison} asserts that $\weps \leq w_{\eps'}$ in $(0,R^n)\times(0,\infty)$. Moreover, since $(\weps)_{\eps\in(0,1)}$ is bounded by $\f m{\omega_n}$, we may infer the existence of a function $w:(0,R^n)\times(0,\infty)\to \R$ such that \eqref{pointwisewepsacc} holds.\\
Furthermore, once again relying on standard parabolic Schauder estimates \cite[Theorem VII.6.1, Theorem VII.5.1]{Ladyzhenskaja1968}, for any $s_0 \in (0,R^n)$ we may find some $\theta\in(0,1)$ with the property that
\begin{align}
(\weps)_{\eps\in(0,1)}~ \text{is bounded in}~C^{2+\theta,1+\f\theta 2}_{loc}([s_0,R^n]\times(0,\infty))
\end{align}
and, by \cite[Theorem V.1.1]{Ladyzhenskaja1968}, possibly after diminishing $\theta \in (0,1)$,
\begin{align}
(\weps)_{\eps\in(0,1)}~\text{is bounded in } C^{\theta,\f\theta 2}_{loc}([s_0,R^n]\times[0,\infty)).
\end{align}
Thereupon, two applications of the Arzel\`{a}-Ascoli theorem and \eqref{pointwisewepsacc} result in
\begin{align*}
w_\eps \to w \qquad \text{in}~C_{loc}^0([s_0,R^n]\times[0,\infty))\cap C_{loc}^{2,1}([s_0,R^n]\times(0,\infty)),
\end{align*}
which, taking $s_0 \searrow 0$, yields \eqref{wepswlimitacc}, and consequently \eqref{limitregacc}. Recalling that $w_{0\eps}(s)=\min\{\f s\eps, w_0(s)\}$, we observe that $w_{0\eps} \to w_0$ in $L^\infty_{loc}((0,R^n))$ as $\eps \searrow 0$, whereupon the topology of convergence warrants that $w$ solves \eqref{0wacc} classically.\\
Finally, \eqref{wepswlimitacc} in conjunction with $w_{\eps s}\geq 0$ entails \eqref{wsgeq0asaprep}.
\end{bew}

This yields the minimal solution of \eqref{0wacc}; moreover, as a preparation for Lemma \ref{strictmonotonacclemma} and Lemma \ref{massaccumulatestotallyacc}, we observe that $w_s$ is positive throughout the interior.

\begin{lemma} \label{minimalsolacc}
Let $n\geq 1$, $R>0$, $m>0$, $\mu := \f{mn}{\omega_n R^n}$, assume \eqref{initacc}, and for $T\in(0,\infty]$ suppose that $w\in C^0((0,R^n]\times [0,T))\cap C^{2,1}((0,R^n]\times (0,T))$ is a nonnegative classical solution of \eqref{0wacc} in $(0,R^n)\times(0,T)$.\\
Then
\begin{align}\label{wbiggerthancreatedlimitsolacc}
w \geq \lim_{\eps \searrow 0} w_\eps \qquad \text{in }(0,R^n]\times[0,T),
\end{align}
where $w_\eps$ is as constructed in Lemma \ref{wepsolacc}, and, if $w_s$ is nonnegative,
\begin{align}\label{wsrealgreaterzeroacc}
w_s>0 \qquad \text{in }(0,R^n)\times(0,T).
\end{align}
\end{lemma}

\begin{bew}
Let $\eps \in (0,1)$ be such that $\eps<\f{R^n\omega_n}m$. We observe that \eqref{fepsacc} implies $f_\eps(\xi) \leq \xi $ for all $\xi \geq 0$. Consequently, the function $\weps$ as in Lemma \ref{wepsolacc} satisfies
\[
\wepst = n^2 s^{2-\frac{2}{n}} \wepsss + n\weps f_\eps(\wepss) - \mu s \wepss \leq n^2 s^{2-\frac{2}{n}} \wepsss + n\weps \wepss - \mu s \wepss
\]
for all $s\in(0,R^n)$ and $t\in(0,T)$. Combined with the regularity of $\weps$ established in Lemma \ref{wepsolacc} and the boundary information, particularly $w_{0\eps} \leq w_0$ in $[0,R^n]$ and $\weps(0,t)=0$ for all $t\in(0,T)$, we may apply Lemma \ref{halfcomparison} to conclude that $\weps \leq w$ in $(0,R^n]\times[0,T)$ for all $\eps \in (0,1)$ with $\eps<\f{R^n\omega_n}m$. Taking the limit $\eps \searrow 0$, this immediately results in \eqref{wbiggerthancreatedlimitsolacc}.\\
In order to confirm \eqref{wsrealgreaterzeroacc}, by standard parabolic theory, we may infer that for any $s_0,s_1 \in (0,R^n)$ with $s_0<s_1$, we have that $z:=w_s$ lies in $L^\infty_{loc}([s_0,s_1]\times[0,T) \cap C^{2,1}([s_0,s_1]\times(0,T))$ and solves
\begin{align*}
z_t &= n^2 s^{2-\f2n} z_{ss} + n^2 \Big(2-\f2n\Big)s^{1-\f2n} z_s + n wz_s + n z^2 - \mu s z_s - \mu z\\
&= n^2 s^{2-\f2n} z_{ss} + \bigg(n^2 \Big(2-\f2n\Big)s^{1-\f2n} + nw - \mu s\bigg)z_s + (nz-\mu)z
\end{align*}
for all $(s,t)\in(s_0,s_1)\times(0,T)$, where the right-hand side is elliptic due to $s\geq s_0>0$. We now let $s_0$ be sufficiently small such that $w_0(s_0)<\f m{\omega_n}$; then we have $z(\cdot,0)=w_s(\cdot,0) \not\equiv 0$ in $[s_0,s_1]$. In light of $w_s$ being nonnegative however, the minimum principle in \cite[Proposition 13.1]{Daners1992AbstractEE} asserts that necessarily $z > 0$ in $(s_0,s_1)\times(0,T)$. Finally letting $s_1 \nearrow R^n$ and $s_0 \searrow 0$, we thus find that indeed $w_s>0$ in $(0,R^n)\times(0,T)$.
\end{bew}

Therefore, we shall call $w:=\lim_{\eps \searrow 0} w_\eps$ the minimal nonnegative solution of \eqref{0wacc}.

\section{Collapse into Dirac singularity}

The key observation toward the collapse into a Dirac singularity will be that $w$ cannot become substantially bigger than $Ks^{-1+\f2n}$ with some $K>0$ in a suitable sense while simultaneously satisfying $w(0,t)=0$. The following lemma lays the groundwork for this conclusion.

\begin{lemma}\label{magicprep}
Let $n\geq 3$, $R>0$, $m>0$ and $\mu \geq 0$. Then for all $q \in(0,1)$, there exists $C=C(q,m,\mu,R,n)>0$ with the property that if for some $t_0\geq 0$ and $\tau>0$, $w\in C^0((0,R^n]\times[t_0,t_0+\tau]) \cap C^{2,1}((0,R^n]\times (t_0,t_0+\tau))$ is a nonnegative classical solution of the boundary value problem in (\ref{0wacc}) with $w_s\geq 0$ in $(0,R^n)\times(t_0,t_0+\tau)$, the inequality
\begin{align}\label{singweight}
\int\limits_{t_0}^{t_0+\tau} \int\limits_0^{R^n} s^{-1+\f2n} w^q(s,t)w_s(s,t) ds dt \leq C(1+\tau)
\end{align}
holds.
\end{lemma}

\begin{bew}
As a preparation, we first observe that according to our assumptions on $w$,
\begin{align}\label{magicprepphi}
\vp(s):=\int\limits_{t_0}^{t_0+\tau} w(s,t) dt, \qquad s\in(0,R^n],
\end{align}
defines a nonnegative and nondecreasing function $\vp \in C^1((0,R^n])$. Necessarily, then there must exist $(s_j)_{j \in \N} \subset (0,R^n)$ such that $s_j \searrow 0$ and
\begin{align}\label{magicprepsj}
s_j \vp_s(s_j) \to 0 \qquad \text{as } j \to \infty,
\end{align}
for otherwise we could find $s_* \in (0,R^n)$ and $c_1>0$ such that $\vp_s(s) \geq \f{c_1}s$ for all $s \in (0,s_*)$, leading to the absurd conclusion that $\vp(s) \leq \vp(s_*) - \ln \f{s_*}s$ for all $s\in(0,s_*)$ and hence $\vp(s)\to -\infty$ as $s \searrow 0$.\\
Upon this choice of $(s_j)_{j\in\N}$, for $j \in \N$ and $\delta \in (0,1)$ we now use the differential equation in (\ref{0wacc}) to obtain the identity
\begin{align}\label{magicprepdiff}
\f1q \f d{dt} \int\limits_{s_j}^{R^n} s^{-1+\f2n} (w+\delta)^q = &n^2 \int\limits_{s_j}^{R^n} s(w+\delta)^{q-1} w_{ss} + n \int\limits_{s_j}^{R^n} s^{-1+\f2n} w(w+\delta)^{q-1}w_s \notag\\
&- \mu \int\limits_{s_j}^{R^n} s^\f2n (w+\delta)^{q-1} w_s \qquad \text{for all } t\in(t_0,t_0+\tau).
\end{align}
Here, three integrations by parts show that thanks to the boundary condition $w(R^n, \cdot) \equiv \f m{\omega_n}$ as well as the nonnegativity of both $w$ and $w_s$,
\begin{eqnarray*}
n^2 \int\limits_{s_j}^{R^n} s(w+\delta)^{q-1} w_{ss} &= &n^2 (1-q) \int\limits_{s_j}^{R^n} s(w+\delta)^{q-2} w_s^2 \\
&&+ n^2 R^n \Big(\f m{\omega_n} + \delta\Big)^{q-1} w_s\Big(\f m{\omega_n},t\Big)-n^2 s_j(w(s_j,t)+\delta)^{q-1} w_s(s_j,t)\\
&& - n^2 \int\limits_{s_j}^{R^n} (w+\delta)^{q-1} w_s \\
& = &n^2 (1-q) \int\limits_{s_j}^{R^n} s(w+\delta)^{q-2} w_s^2 \\
&&+ n^2 R^n \Big(\f m{\omega_n} + \delta\Big)^{q-1} w_s\Big(\f m{\omega_n},t\Big)-n^2 s_j(w(s_j,t)+\delta)^{q-1} w_s(s_j,t)\\
&& - \f{n^2}q\Big(\f m{\omega_n} + \delta\Big)^{q} + \f{n^2}q (w(s_j,t)+ \delta)^q\\
&\geq& -n^2 s_j (w(s_j,t)+\delta)^{q-1} w_s(s_j,t) - \f{n^2}q \Big(\f m{\omega_n}+\delta \Big)^q \qquad \text{for all } t \in (t_0,t_0+\tau)
\end{eqnarray*}
and
\begin{eqnarray*}
-\mu \int\limits_{s_j}^{R^n} s^\f2n (w+\delta)^{q-1} w_s &=& \f{2\mu}{nq} \int\limits_{s_j}^{R^n} s^{-1+\f2n}(w+\delta)^q - \f\mu q R^2 \Big(\f m{\omega_n}+\delta\Big)^q + \f\mu q s_j^\f2n (w(s_j,t)+\delta)^q\\
&\geq& - \f\mu q R^2 \Big(\f m{\omega_n}+\delta\Big)^q \qquad \text{for all } t \in (t_0,t_0+\tau).
\end{eqnarray*}
Therefore, upon an integration in time, (\ref{magicprepdiff}) entails that
\begin{eqnarray}\label{magicpreptimeint}
\hspace*{-10mm} & &n \int\limits_{t_0}^{t_0+\tau} \int\limits_{s_j}^{R^n} s^{-1+\f2n} w(s,t)(w(s,t)+\delta)^{q-1} w_s(s,t) ds dt \\
&\leq& \f1q \int\limits_{s_j}^{R^n} s^{-1+\f2n} (w(s,t_0+\tau)+\delta)^q ds - \f1q \int\limits_{s_j}^{R^n} s^{-1+\f2n} (w(s,t_0)+\delta)^q ds \notag\\
&& + n^2 s_j \int_{t_0}^{t_0+\tau} (w(s_j,t)+\delta)^{q-1} w_s(s_j,t) dt + \f{n^2}q \Big(\f m{\omega_n}+\delta\Big)^q \cdot \tau + \f{\mu R^2}{q} \Big(\f m{\omega_n}+\delta\Big)^q \cdot \tau 
\end{eqnarray}
for all $j\in\N$. Here since $w\leq \f m{\omega_n}$,
\begin{align*}
\f1q \int\limits_{s_j}^{R^n} s^{-1+\f2n} (w(s,t_0+\tau)+\delta)^q ds &\leq \f1q \Big(\f m{\omega_n}+\delta\Big)^q \cdot \int\limits_0^{R^n} s^{-1+\f2n} ds\\
&= \f{nR^2}{2q}\Big(\f m{\omega_n}+\delta\Big)^q \qquad \text{for all } j \in \N,
\end{align*}
whereas combining the nonnegativity of $w$ with the fact that $q\leq 1$ yields
\begin{align*}
n^2 s_j \int\limits_{t_0}^{t_0+\tau} (w(s_j,t)+\delta)^{q-1} w_s(s_j,t) dt &\leq n^2 \delta^{q-1} s_j \cdot \int\limits_{t_0}^{t_0+\tau} w_s(s_j,t) dt\\
&= n^2 \delta^{q-1} \cdot s_j \vp_s(s_j) \qquad \text{for all } j\in\N
\end{align*}
according to our definition (\ref{magicprepphi}) of $\vp$. In light of (\ref{magicprepsj}), from (\ref{magicpreptimeint}) on neglecting  the second summand therein and taking $j\to \infty$ we thus infer that as a consequence of Fatou's lemma,
\begin{eqnarray*}
\hspace*{-10mm} & & n \int\limits_{t_0}^{t_0+\tau} \int\limits_{0}^{R^n} s^{-1+\f2n} w(s,t)(w(s,t)+\delta)^{q-1} w_s(s,t) ds dt \\
&\leq & \f{nR^2}{2q}\Big(\f m{\omega_n}+\delta\Big)^q + \f{n^2}q \Big(\f m{\omega_n}+\delta\Big)^q \cdot \tau + \f{\mu R^2}{q} \Big(\f m{\omega_n}+\delta\Big)^q \cdot \tau 
\end{eqnarray*}
for any $\delta  \in (0,1)$. Again due to Fatou's lemma, letting $\delta \searrow 0$ hence results in (\ref{singweight}).
\end{bew}

Now we can prove the previously indicated statement, which in its precise form reads as follows:

\begin{lemma}\label{magic}
Let $n\geq 3$, $R>0$, $m>0$, $\mu \geq 0$, $\gamma \in (0,1-\f2n)$, $t_0\geq 0$ and $\tau>0$. Then there exists $C(m,\gamma,\tau,R,n)>0$ with the property that any nonnegative solution $w\in C^0((0,R^n]\times[t_0,t_0+\tau]) \cap C^{2,1}((0,R^n]\times (t_0,t_0+\tau))$ of (\ref{0wacc}) with $w_s\geq 0$ in $(0,R^n)\times(t_0,t_0+\tau)$ satisfies
\begin{align}\label{w00acc}
w(0,t)=0 \qquad \text{for all } t \in (t_0,t_0+\tau),
\end{align}
then
\begin{align}\label{magicbound}
\int\limits_{t_0}^{t_0+\tau} \vertiii{w(\cdot,t)}_{\gamma} dt \leq C(m,\gamma,\tau, R, n),
\end{align}
where $\vertiii{\varphi}_\gamma \coloneq \sup\limits_{s\in(0,R^n)} \{s^{-\gamma}|\varphi(s)|\}$ for $\varphi \in C^0((0,R^n]) \cap L^\infty((0,R^n))$.
\end{lemma}

\begin{bew}
By applying Lemma \ref{magicprep}, for all $q\in(0,1)$ we obtain the existence of a constant $C=C(q,m,\mu,R,n)>0$ such that
\begin{align*}
\int\limits_{t_0}^{t_0+\tau} \int\limits_0^{R^n} s^{-1+\f2n} w^q(s,t)w_s(s,t) ds dt \leq C(1+\tau).
\end{align*}
Together with (\ref{w00acc}), for any $\eta>0$ and $s\in(0,R^n)$, we have
\begin{align*}
\int\limits_{t_0}^{t_0+\tau} (s+\eta)^{-1+\f2n}  w^{q+1}(s,t) dt &= \eta^{-1+\f2n} \int\limits_{t_0}^{t_0+\tau} w^{q+1}(0,t) dt +\int\limits_{t_0}^{t_0+\tau}\int\limits_0^s \partial_\sigma \{(\sigma+\eta)^{-1+\f2n} w^{q+1}(\sigma,t) \}d\sigma dt\\
&\leq (q+1) \int\limits_{t_0}^{t_0+\tau}\int\limits_0^s (\sigma + \eta)^{-1+\f2n} w^q w_s d\sigma dt\\
&\leq (q+1) C(1+\tau),
\end{align*}
which, upon letting $\eta \searrow 0$, results in
\[
\int\limits_{t_0}^{t_0+\tau} s^{-1+\f2n} w^{q+1}(s,t) dt \leq K,
\]
where $K=K(q,R,m,\mu,n,\tau):= (q+1) C(1+\tau)$.\\
In order to appropriately utilize this inequality to acquire (\ref{magicbound}), we note that for any $\gamma \in (0,1-\f2n)$, the inequality $w > s^{\gamma}$ is equivalent to $w^{\f{1-\f2n-\gamma}{\gamma}}>s^{1-\f2n-\gamma}$, and thus for $\gamma>\frac{n-2}{2n}$, with $q=q(\gamma,n):=\f{1-\f2n-\gamma}{\gamma}<1$, we infer
\begin{align*}
\int\limits_{t_0}^{t_0+\tau} s^{-\gamma}w(s,t) dt &= \int\limits_{\substack{t_0\\\{w\leq s^\gamma\}}}^{t_0+\tau} s^{-\gamma}w(s,t) dt + \int\limits_{\substack{t_0\\\{w> s^\gamma\}}}^{t_0+\tau} s^{-\gamma}w(s,t) dt\\
&\leq \int\limits_{\substack{t_0\\\{w\leq s^\gamma\}}}^{t_0+\tau} 1 dt + \int\limits_{\substack{t_0\\\{w> s^\gamma\}}}^{t_0+\tau} s^{-\gamma}w(s,t) s^{-1+\f2n+\gamma} w^{\f{1-\f2n-\gamma}{\gamma}} dt\\
&\leq \tau +  \int\limits_{\substack{t_0\\\{w> s^\gamma\}}}^{t_0+\tau} s^{-1+\f2n}w^{q+1}(s,t)  dt\\
&\leq \tau +  \int\limits_{t_0}^{t_0+\tau} s^{-1+\f2n}w^{q+1}(s,t)  dt\\
&\leq \tau + K.
\end{align*}
For $\gamma\in(0,\f{n-2}{2n}]$, we may trivially reach the same conclusion by estimating 
\[
s^{-\gamma} \leq C \cdot s^{-\alpha}
\]
for some $\alpha \in (\f{n-2}{2n},1-\f2n)$ and $C=C(R,n)$. Since $K$ is independent of $s$, this readily results in (\ref{magicbound}).
\end{bew}

Consequently, in order to violate \eqref{magicbound}, we seek subsolutions which are sufficiently close to $cs^{-\gamma}$ for some $c>0$ and $\gamma \in(0,1-\f2n)$. One such family is given by

\begin{lemma}\label{compfctmassacc}
Let $n\geq 3$, $R>0$, $\gamma \in(0,1-\f2n)$, $m>0$, $\mu\geq 0$ and $y_0>0$, and let
\begin{align}\label{betadefacc}
\beta \coloneq 1 - \f2{n-2-n\gamma}
\end{align}
Then there exists $C>0$ such that for all $\delta>0$, writing $T \coloneq \f{y_0^{1-\beta}}{(1-\beta)C} \in (0,\infty]$, we have the following:\\
The function $\uw_\delta$ defined on $[0,R^n]\times[0,T)$ by letting
\begin{align*}
\uw_\delta(s,t):= y(t) \cdot \varphi_\delta(s), \qquad s\in[0,R^n],~t\in [0,T),
\end{align*}
where
\begin{align*}
\varphi_\delta(s) \coloneq \f{s^2}{s^{2-\gamma}+\delta}, \qquad s\in[0,R^n],
\end{align*}
and $y\in C^2([0,T))$ is denoting the solution of
\begin{align}\label{ydefacc}
\left\lbrace
\begin{array}{r@{}l@{\quad}l@{\quad}l@{\,}c}
&y'(t) = -C y^\beta(t), \quad t\in[0,T),\\
&y(0)=y_0,
\end{array}\right.
\end{align}
that is,
\begin{align}\label{yexplacc}
y(t)= \{y_0^{1-\beta}-(1-\beta)Ct\}^{\f1{1-\beta}}, \qquad t\in[0,T),
\end{align}
belongs to $C^2([0,R^n]\times[0,T))$ and satisfies
\begin{align}\label{uwdiffineqacc}
\uw_{\delta t} \leq n^2 s^{2-\f2n}\uw_{\delta ss} + n \uw_\delta \uw_{\delta s} - \mu s \uw_{\delta s} \qquad \text{in}~[0,R^n]\times(0,T).
\end{align}
\end{lemma}

\begin{bew}
For all $s\in(0,R^n)$, we firstly compute
\begin{align*}
\vp_{\delta s}(s) = \f{2s(s^{2-\gamma}+\delta)-s^2(2-\gamma)s^{1-\gamma}}{(s^{2-\gamma}+\delta)^2}=\f{\gamma s^{3-\gamma}+2\delta s}{(s^{2-\gamma}+\delta)^2}
\end{align*}
and
\begin{align*}
\vp_{\delta ss} (s) &= \f{(\gamma(3-\gamma)s^{2-\gamma}+2\delta)(s^{2-\gamma}+\delta)^2 - (\gamma s^{3-\gamma}+2\delta s)\cdot 2(s^{2-\gamma}+\delta)(2-\gamma)s^{1-\gamma}}{(s^{2-\gamma}+\delta)^4}\\
&= \f{\gamma(3-\gamma)s^{4-2\gamma}+\gamma(3-\gamma)\delta s^{2-\gamma}+2\delta s^{2-\gamma} + 2\delta^2 - 2 \gamma (2-\gamma)s^{4-2\gamma} - 4(2-\gamma)\delta s^{2-\gamma}}{(s^{2-\gamma}+\delta)^3}\\
&= \f{\{\gamma(3-\gamma)-2\gamma(2-\gamma)\}s^{4-2\gamma}+\{\gamma(3-\gamma)+2-4(2-\gamma)\}\delta s^{2-\gamma} +2\delta^2 }{(s^{2-\gamma}+\delta)^3}\\
&= \f{-\gamma(1-\gamma)s^{4-2\gamma}-(1-\gamma)(6-\gamma)\delta s^{2-\gamma} + 2\delta^2 }{(s^{2-\gamma}+\delta)^3},\\
\end{align*}
by which, via continuous extension, $\vp \in C^2([0,R^n])$, because $\gamma<2$. 
Picking 
\begin{equation}\label{Cdefacc}
C:=\max \Big\{4\mu y_0^{1-\beta}, 6n^2(1-\gamma) (3n(1-\gamma))^{\f{2}{n-2-n\gamma}}\Big\},
\end{equation}
we observe that then
\[
4\mu y \leq Cy^\beta
\]
in $(0,T)$ due to (\ref{yexplacc}) and hence, by (\ref{ydefacc}),
\begin{align}\label{mutermacc}
\f12 y' \vp_\delta +  \mu s \uw_{\delta s} &= \f12 y' \f{s^2}{s^{2-\gamma}+\delta} +  \mu s y \f{\gamma s^{3-\gamma}+2\delta s}{(s^{2-\gamma}+\delta)^2} \notag\\
&=\f{\f12 y's^2(s^{2-\gamma}+\delta)+ \mu s y(\gamma s^{3-\gamma}+2\delta s)}{(s^{2-\gamma}+\delta)^2} \notag\\
&=\f{s^2(-\f12 Cy^\beta s^{2-\gamma}-\f12 Cy^\beta \delta+ \mu  y\gamma s^{2-\gamma}+2\delta \mu y )}{(s^{2-\gamma}+\delta)^2} \notag\\
&\leq \f{s^2(-2\mu y s^{2-\gamma}-2\mu y \delta+ \mu  y\gamma s^{2-\gamma}+2\delta \mu y )}{(s^{2-\gamma}+\delta)^2}\notag\\
&= \f{-\mu y s^2 (2-\gamma) s^{2-\gamma}}{(s^{2-\gamma}+\delta)^2} \notag\\
&\leq 0
\end{align}
for all $(s,t)\in[0,R^n]\times(0,T)$.\\
Now fixing any $t>0$ and thereupon defining
\begin{equation}\label{s0defacc}
s_0:=\Big(\f1{3n(1-\gamma)}\Big)^\f1{1-\f2n-\gamma} \cdot y^\f1{1-\f2n-\gamma}(t),
\end{equation}
we may establish that for $s \leq s_0$
\begin{align*}
\f{-nys^{2-\f2n}\vp_{\delta ss}}{y^2 \vp_\delta \vp_{\delta s}} &= \f{-ns^{2-\f2n}(-\gamma(1-\gamma)s^{4-2\gamma}-(1-\gamma)(6-\gamma)\delta s^{2-\gamma} + 2\delta^2)}{ys^2(\gamma s^{3-\gamma}+2\delta s)}\\
&\leq \f{ns^{2-\f2n}(\gamma(1-\gamma)s^{4-2\gamma}+(1-\gamma)(6-\gamma)\delta s^{2-\gamma})}{ys^2(\gamma s^{3-\gamma}+2\delta s)}\\
&=\f{ns^{1-\f2n-\gamma}}{y}\cdot \f{\gamma(1-\gamma)s^{5-\gamma}+(1-\gamma)(6-\gamma)\delta s^{3}} {\gamma s^{5-\gamma}+2\delta s^3}\\
&= \f{ns^{1-\f2n-\gamma}(1-\gamma)}{y}\cdot \f{\gamma s^{2-\gamma}+(6-\gamma)\delta} {\gamma s^{2-\gamma}+2\delta}\\
&\leq \f{ns^{1-\f2n-\gamma}(1-\gamma)}{y}\cdot \f{3\gamma s^{2-\gamma}+6\delta} {\gamma s^{2-\gamma}+2\delta}\\
&= \f{ns^{1-\f2n-\gamma}(1-\gamma)}{y}\cdot 3\\
&\leq \f{3n(1-\gamma)}{y}\cdot s_0^{1-\f2n-\gamma}\\
&= 1,
\end{align*}
and consequently, together with (\ref{mutermacc}),
\begin{align}\label{smallsineqacc}
\uw_{\delta t} - n^2 s^{2-\f2n}\uw_{\delta ss} - n \uw_\delta \uw_{\delta s} + \mu s \uw_{\delta s} &=y' \vp_\delta - n^2 s^{2-\f2n}\uw_{\delta ss} - n \uw_\delta \uw_{\delta s} + \mu s \uw_{\delta s} \notag\\
&\leq - n^2 s^{2-\f2n}\uw_{\delta ss} - n \uw_\delta \uw_{\delta s} \notag\\
&= n(-nys^{2-\f2n}\vp_{\delta ss}-y^2 \vp_\delta \vp_{\delta s}) \notag\\
&\leq 0.
\end{align}
If on the other hand $s>s_0$, we estimate
\begin{align*}
\f{n^2 s^{2-\f2n}\uw_{\delta ss}}{\f12 y' \vp_\delta} &=2n^2 s^{2-\f2n}y^{1-\beta}\f{\gamma(1-\gamma)s^{4-2\gamma}+(1-\gamma)(6-\gamma)\delta s^{2-\gamma} - 2\delta^2 }{Cs^2(s^{2-\gamma}+\delta)^2}\\
&\leq 2n^2 s^{2-\f2n}y^{1-\beta}\f{\gamma(1-\gamma)s^{4-2\gamma}+(1-\gamma)(6-\gamma)\delta s^{2-\gamma}}{Cs^2(s^{2-\gamma}+\delta)^2}\\
&= \f{2n^2(1-\gamma)}C s^{-\f2n}y^{1-\beta}\f{\gamma s^{4-2\gamma}+(6-\gamma)\delta s^{2-\gamma}}{s^{4-2\gamma}+2\delta s^{2-\gamma}+\delta^2}\\
&\leq \f{2n^2(1-\gamma)}C s^{-\f2n}y^{1-\beta}\f{3 s^{4-2\gamma}+6\delta s^{2-\gamma}}{s^{4-2\gamma}+2\delta s^{2-\gamma}}\\
&= \f{6n^2(1-\gamma)}C s^{-\f2n}y^{1-\beta}\\
&= \f{6n^2(1-\gamma)}C s_0^{-\f2n}y^{1-\beta}\\
&= \f{6n^2(1-\gamma)}C \Big(\f1{3n(1-\gamma)}\Big)^{-\f{2}{n-2-n\gamma}} \cdot y^{-\f{2}{n-2-n\gamma}+1-\beta}\\
&= \f{6n^2(1-\gamma)}C \Big(\f1{3n(1-\gamma)}\Big)^{-\f{2}{n-2-n\gamma}}\\
&\leq 1
\end{align*}
in for all $s \in [0,R^n]$ and $t\in(0,T)$ by means of (\ref{s0defacc}), (\ref{betadefacc}) and (\ref{Cdefacc}). In view of the nonpositivity of $y'$, this implies that $\f12 y' \vp_\delta \leq n^2 s^{2-\f2n}\uw_{\delta ss}$, and accordingly, once more relying on (\ref{mutermacc}),
\begin{align}\label{largesineqacc}
\uw_{\delta t} - n^2 s^{2-\f2n}\uw_{\delta ss} - n \uw_\delta \uw_{\delta s} + \mu s \uw_{\delta s} &=y' \vp_\delta - n^2 s^{2-\f2n}\uw_{\delta ss} - n \uw_\delta \uw_{\delta s} + \mu s \uw_{\delta s} \notag\\
&\leq \f12 y' \vp_\delta - n^2 s^{2-\f2n}\uw_{\delta ss}  \notag\\
&\leq 0
\end{align}
in $[0,R^n]\times(0,T)$. We may now simply combine (\ref{smallsineqacc}) and (\ref{largesineqacc}) to arrive at (\ref{uwdiffineqacc}).
\end{bew}

By choosing $\delta>0$ sufficiently small, we obtain a contradiction to \eqref{magicbound} for $\gamma_0\in(\gamma,1-\f2n)$ and thereby establish that Dirac formation can not only occur but happen arbitrarily quickly.

\begin{lemma}\label{estwg0acc}
Let $n\geq 3$, $R>0$, $m>0$, $\gamma \in (0,1-\f2n)$, $\tau>0$ and $y_0>0$. Then one can find $\delta_* = \delta_*(n,m,\gamma,\tau,y_0) \in (0,1)$ with the property that whenever 
$w_0$ satisfies \eqref{initacc} as well as
\begin{align}\label{initialconditiondiracmassw}
w_0(s) \geq y_0\cdot \frac{s^2}{s^{2-\gamma}+\delta_*} \qquad \text{for all }s\in[0,R^n],
\end{align}
any nonnegative solution $w\in C^0((0,R^n]\times [0,\infty)) \cap C^{2,1}((0,R^n]\times(0,\infty))$ of \eqref{0wacc} with $w_s \geq 0$ in $(0,R^n]\times(0,\tau)$ has the property that
\begin{align}\label{diracmassw}
w(0,t_0)>0 \qquad \text{for some}~t_0\in(0,\tau).
\end{align}
\end{lemma}

\begin{bew}
We first set $\beta$ as in (\ref{betadefacc}), $T \coloneq \f{y_0^{1-\beta}}{(1-\beta)C}$, and let $y \in C^2([0,T))$ denote the solution of
\begin{align*}
\left\lbrace
\begin{array}{r@{}l@{\quad}l@{\quad}l@{\,}c}
&y'(t) = -C y^\beta(t), \quad t\in[0,T),\\
&y(0)=y_0.
\end{array}\right.
\end{align*}
Thereupon, we fix $\tau\in(0,T)$ and $c_1 \in(0,y_0)$ such that $y(t)\geq c_1$ for all $t\in[0,\tau]$.\\
Let $w \in C^0((0,R^n]\times [0,\infty)) \cap C^{2,1}((0,R^n]\times(0,\infty))$ be a nonnegative solution to (\ref{0wacc}) with $w_0$ satisfying \eqref{w0acc}.
If we now suppose that furthermore
\begin{align}\label{zerohypothesisw}
w(0,t)=0 \qquad \text{for all }t\in[0,\tau],
\end{align}
then after fixing any $\gamma_0 \in (\gamma,1-\f2n)$, Lemma \ref{magic} becomes applicable so as to warrant that
\begin{align}\label{wnormintest}
\int\limits_{0}^{\tau} \vertiii{w(\cdot,t)}_{\gamma_0} dt \leq c_2
\end{align}
for some $c_2 = c_2(m,\gamma_0,\tau, R, n)$ but, importantly, independent of the particular choice of $w_0$.\\
Defining
\begin{align}\label{c3defgammabuacc}
c_3:=c_1 (\gamma_0-\gamma)^{-\frac{\gamma_0-\gamma}{2-\gamma}}\cdot \frac{(2-\gamma_0)^\frac{2-\gamma_0}{2-\gamma}}{2-\gamma},
\end{align}
we pick $\delta>0$ such that
\begin{align}\label{deltachoice34acc}
 \delta < \min\bigg\{\frac{\gamma_0-\gamma}{2-\gamma_0}R^{n(2-\gamma)}, \Big(\f{c_3 \tau}{c_2}\Big)^{\frac{2-\gamma}{\gamma_0-\gamma}}\bigg\},
\end{align}
and assume that furthermore, $w_0$ satisfies the inequality
\[
w_{0}(s) \geq y_0\cdot \frac{s^2}{s^{2-\gamma}+\delta} \qquad \text{for all }s\in[0,R^n].
\]
Since moreover, $w_{0}(0)=0$ and $w_{0s}$ is bounded, we may draw upon Lemma \ref{compfctmassacc} and our comparison principle \ref{halfcomparison} to establish that $w$ satisfies the inequality
\begin{equation*}
w \geq y \cdot \frac{s^2}{s^{2-\gamma}+\delta} \geq c_1 \cdot \frac{s^2}{s^{2-\gamma}+\delta} \qquad \text{in }(0,R^n]\times[0,\tau].
\end{equation*}
Consequently, we may infer that
\begin{align}\label{vertiiiest1}
\vertiii{w(\cdot,t)}_{\gamma_0} &= \sup_{s\in(0,R^n)} \{s^{-\gamma_0} w_\delta(s,t) \}\notag\\
&\geq c_1 \sup_{s\in(0,R^n)} \Big\{s^{-\gamma_0} \frac{s^2}{s^{2-\gamma}+\delta} \Big\}\notag\\
&= c_1 \sup_{s\in(0,R^n)} \frac{s^{2-\gamma_0}}{s^{2-\gamma}+\delta}
\end{align}
for $t\in [0,\tau]$. A search for local extrema then yields that $\Phi(s):=\frac{s^{2-\gamma_0}}{s^{2-\gamma}+\delta}$ attains a local maximum at $s_0:=\Big(\frac{(2-\gamma_0)\delta}{\gamma_0-\gamma}\Big)^\frac1{2-\gamma}$, so that since $\delta < \frac{\gamma_0-\gamma}{2-\gamma_0}R^{n(2-\gamma)}$ by \eqref{deltachoice34acc}, we have $s_0\in(0,R^n)$ and thus
\begin{align}\label{vertiiiest2}
\sup_{s\in(0,R^n)} \frac{s^{2-\gamma_0}}{s^{2-\gamma}+\delta} &\geq \frac{s_0^{2-\gamma_0}}{s_0^{2-\gamma}+\delta}\notag\\
&=\frac{\Big(\frac{(2-\gamma_0)\delta}{\gamma_0-\gamma}\Big)^\frac{2-\gamma_0}{2-\gamma}}{\frac{(2-\gamma_0)\delta}{\gamma_0-\gamma}+\delta}\notag\\
&=\frac{\Big(\frac{2-\gamma_0}{\gamma_0-\gamma}\Big)^\frac{2-\gamma_0}{2-\gamma}}{\frac{2-\gamma}{\gamma_0-\gamma}}\cdot \delta^{-\frac{\gamma_0-\gamma}{2-\gamma}}\notag\\
&=(\gamma_0-\gamma)^{-\frac{\gamma_0-\gamma}{2-\gamma}}\cdot \frac{(2-\gamma_0)^\frac{2-\gamma_0}{2-\gamma}}{2-\gamma}\cdot \delta^{-\frac{\gamma_0-\gamma}{2-\gamma}}.
\end{align}
Combining (\ref{vertiiiest1}) and (\ref{vertiiiest2}) results in
\begin{align*}
\vertiii{w(\cdot,t)}_{\gamma_0} \geq c_3 \delta^{-\frac{\gamma_0-\gamma}{2-\gamma}} \qquad \text{for all }t \in[0,\tau]
\end{align*}
with $c_3$ as in \eqref{c3defgammabuacc} and hence, together with (\ref{wnormintest}),
\begin{align*}
c_2 \geq \int\limits_{0}^{\tau} \vertiii{w(\cdot,t)}_{\gamma_0} dt \geq c_3 \tau \delta^{-\frac{\gamma_0-\gamma}{2-\gamma}}.
\end{align*}
This implies that
\[
\delta \geq \Big(\f{c_3 \tau}{c_2}\Big)^{\frac{2-\gamma}{\gamma_0-\gamma}},
\]
contradicting \eqref{deltachoice34acc}. Therefore, (\ref{zerohypothesisw}) must be false.
Accordingly, if (\ref{initialconditiondiracmassw}) is satisfied, we may conclude (\ref{diracmassw}).
\end{bew}

On the other hand, we can acquire an upper bound for $w(0,t)$ by comparing to the approximating solutions $\weps$ and utilizing the monotone convergence property $\weps \nearrow w$ as $\eps \searrow 0$. Note that \eqref{initialconditionslowcollapsew} can be satisfied for any $w_0\in C^1([0,R^n])$ with $w_0(0)=0$ if we pick $a>0$ adequately large.

\begin{lemma}\label{slowcollapseacc}
Let $n\geq 1$, $R>0$, $m>0$, $\mu := \f{mn}{\omega_n R^n}$, $a>0$, $\gamma \in (0,1]$ and suppose $w_0$ satisfies (\ref{initacc}) and is bounded via
\begin{align}\label{initialconditionslowcollapsew}
w_0(s) \leq (as)^\gamma \qquad \text{for all }s\in[0,R^n].
\end{align}
Then for any $\tau_*>0$ there exists $c=c(\tau_*)>0$ such that $w$ as obtained in Lemma \ref{createdsol} satisfies
\begin{align}\label{slowcollapsew}
w(0,t) \leq ct^\gamma \qquad \text{for all}~t \in (0, \tau_*).
\end{align}
\end{lemma}

\begin{bew}
Let $y\in C^1([0,\infty))$ be the solution of
\begin{align*}
\left\lbrace
\begin{array}{r@{}l@{\quad}l@{\quad}l@{\,}c}
&y'(t) = an(y(t)+aR^n)^\gamma, \quad t>0,\\
&y(0) = 0.
\end{array}\right.
\end{align*}
Thereupon we introduce $\ow \in C^0([0,R^n]\times[0,\infty)) \cap C^{2,1}((0,R^n]\times[0,\infty))$ defined by
\[
\ow(s,t):=(y(t)+as)^\gamma, \qquad s\in[0,R^n],~t\geq 0.
\]
Recalling \eqref{fepsacc}, for any $\eps \in (0,1)$ a straightforward calculation now yields
\begin{eqnarray*}
\ow_t - n^2 s^{2-\f2n} \ow_{ss} - n \ow f_\eps(\ow_s) + \mu s \ow_s &=& \gamma (y+as)^{\gamma-1} y' + n^2 a^2 \gamma (1-\gamma) s^{2-\f2n} (y+as)^{\gamma-2}\\
&& - na\gamma (y+as)^\gamma f_\eps((y+as)^{\gamma-1}) + \mu a \gamma s (y+as)^{\gamma-1}\\
&\geq & \gamma (y+as)^{\gamma-1} y' + n^2 a^2 \gamma (1-\gamma) s^{2-\f2n} (y+as)^{\gamma-2}\\
&& - na\gamma (y+as)^{2\gamma-1} + \mu a \gamma s (y+as)^{\gamma-1}
\end{eqnarray*}
for all $s\in(0,R^n)$ and $t>0$, where the second and the fourth summand on the right-hand side are nonnegative. Consequently, we may further estimate
\begin{eqnarray*}
\ow_t - n^2 s^{2-\f2n} \ow_{ss} - n \ow f_\eps(\ow_s) + \mu s \ow_s &\geq & \gamma (y+as)^{\gamma-1} y'  - na\gamma (y+as)^{2\gamma-1}\\
&=& na\gamma (y+as)^{\gamma-1} (y+aR^n)^\gamma - na\gamma (y+as)^{2\gamma-1}\\
&\geq& 0.
\end{eqnarray*}
Obviously, $\ow$ is nonnegative and for $\eps\in (0,\f{R^n \omega_n}{m})$, the corresponding solutions $\weps$ of \eqref{0wepsacc} from Lemma \ref{wepsolacc} lie in $C^0([0,R^n]\times[0,\infty)) \cap C^{2,1}((0,R^n]\times(0,\infty))$ and satisfy $\weps(0,t)=0$ for all $t \geq 0$ as well as 
\[
w_{0\eps}(s)=\min\Big\{\f s\eps, w_0(s)\Big\} \leq w_0(s) \leq (as)^\gamma = \ow(s,0) \qquad \text{for all }s\in[0,R^n]
\]
by \eqref{initialconditionslowcollapsew}.
Hence we may utilize Lemma \ref{halfcomparison} to infer that
\[
\weps(s,t)\leq \ow(s,t) \qquad \text{for all }s\in(0,R^n] \times[0,\infty)
\]
and for any $\eps \in (0,1)$ with $\eps<\f{R^n \omega_n}{m}$. Passing to the limit $\eps \searrow 0$ subsequently results in
\begin{align}\label{slowcollapsecomp}
w(s,t)\leq \ow(s,t) \qquad \text{for all }s\in(0,R^n] \times[0,\infty).
\end{align}
Now for any fixed $\tau_*>0$, due to the continuity of $y$ there exists $c_1>0$ such that $y(t)\leq c_1$ in $[0,\tau_*]$, and thus 
\[
y'(t)\leq an(c_1+aR^n)^\gamma =: c_2 \qquad \text{in }[0,\tau_*].
\]
Accordingly, as $y(0)=0$, we have $y(t)\leq c_2 t$ for all $t\in[0,\tau_*]$.\\
Setting $c:=c_2^\gamma$, together with \eqref{slowcollapsecomp} we arrive at \eqref{slowcollapsew}.
\end{bew}

Hence, in stark contrast to known mass accumulation phenomena in two dimensions (\cite{Biler2008}), no instantaneous aggregation of some positive mass but rather a slow collapse into a Dirac singularity takes place here.

\begin{corollary}\label{corslowcollapseacc}
Let $n\geq 3$, $R>0$, $m>0$, $\mu = \f{mn}{\omega_n R^n}$, $\gamma_i \in(0,1-\f2n)$ for $i\in\{1,2\}$, $\uc > 0,~\oc > 0$, $\tau>0$ and $\eta>0$. Then there exists $\delta>0$ such that if $w_0$ is as in \eqref{initacc} and complies with
\begin{align}\label{doubleboundcoracc}
\uc \cdot \f{s^2}{s^{2-\gamma_1}+\delta} \leq w_0(s) \leq \oc \cdot s^{\gamma_2} \qquad \text{for all }s\in[0,R^n],
\end{align}
then $w$ as constructed in Lemma \ref{createdsol} has the property that
\begin{align}\label{collapsebutslowacc}
0<w(0,t_0)<\eta \qquad \text{for some }t_0\in (0,\tau).
\end{align}
\end{corollary}

\begin{bew}
Due to the second inequality in \eqref{doubleboundcoracc}, for any given $\eta>0$ we may draw upon Lemma \ref{slowcollapseacc} to pick $\tau_* \leq \tau$ small enough to obey
\begin{align}\label{corcollapseproof1}
w(0,t) \leq ct^\gamma < \eta \qquad \text{for all}~t \in (0, \tau_*).
\end{align}
Thereupon, now utilizing the lower bound imposed on $w_0$ in \eqref{doubleboundcoracc}, therein we pick $\delta = \delta(n,m,\gamma_1,\tau_*,\uc) \in (0,1)$ as in Lemma \ref{estwg0acc} in such a manner that
\begin{align}\label{corcollapseproof2}
w(0,t_0) > 0 \qquad \text{for some }t_0 \in (0,\tau_*).
\end{align}
Merging \eqref{corcollapseproof1} and \eqref{corcollapseproof2}, we arrive at \eqref{collapsebutslowacc}.
\end{bew}

\section{Monotonicity and right-continuity of $w(0,\cdot)$}

\subsection{Monotonicity}\label{subsectionmonacc}

Our first aim in this section is to prove that $w(0,\cdot)$ is monotone increasing, thereby establishing the irreversibility of mass concentration. If for some $t_0\geq 0$ however, $w(0,t_0)$ is positive, we have no means to simply compare it to a positive constant from below, as that would require precisely the conclusion we are trying to prove as an a priori information in \eqref{uwowgeqaacc} of Lemma \ref{halfcomparison}. Instead, we shall rely on the following construction.
Let $n\geq 1$, $m>0$, 
\begin{lemma}\label{staticsubsolacc}
Let $n\geq 3$, $R>0$, $m>0$ and $\mu := \f{mn}{\omega_n R^n}$. Then for any choice of positive numbers $s_0,s_1,M_0$ and $M_1$ fulfilling
\begin{align}\label{s0s1M0M1}
0<s_0<s_1<R^n \quad \text{and} \quad 0<M_0<M_1<\f m{\omega_n},
\end{align}
one can find $s_*\leq s_0$ and a function 
\[
\uW = \uW^{(s_0,s_1,M_0,M_1)} \in W^{2,\infty}((0,R^n)) \cap C^\infty([0,R^n]\setminus\{s_*,s_1\})
\]
such that $\uW(0)=0$ and $\uW(R^n)=\f m{\omega_n}$, that
\begin{align}\label{s0M0s1M1}
\uW(s_0) \geq M_0 \quad \text{and} \quad \uW(s_1)\leq M_1,
\end{align}
that $\uW_s>0$ on $[0,R^n]$, that $\uW_{ss}\geq 0$ in $(s_1,R^n)$, and that
\begin{align}\label{diffineqexcepts*s1acc}
-n^2 s^{2-\f2n}\uW_{ss} - n \uW\uW_s + \mu s \uW_s \leq 0 \qquad \text{for all }s\in(0,R^n)\setminus\{s_*,s_1\}.
\end{align}
\end{lemma}

\begin{bew}
Given $m>0$, $\mu = \f{mn}{\omega_n R^n}$ and $s_0, s_1, M_0,M_1$ as in \eqref{s0s1M0M1}, we may pick $s_* \in (0,R^n)$ small enough such that
\begin{align}\label{s*smallers0acc}
s_* \leq s_0
\end{align}
as well as
\begin{align}\label{s*smallenoughacc}
2n s_*^{1-\f2n} + \f{\mu}n (s_* +s_*^2)\leq M_0,
\end{align}
where in achieving the latter, we make use of our assumption that $n\geq 3$. Thereupon, we fix $\lambda>0$ satisfying
\begin{align}\label{lambdas*acc}
\lambda \geq \f{\mu}{n^2 s_*^{1-\f2n}},
\end{align}
and with
\begin{align}\label{kappa0defacc}
\kappa_0:= \f\mu{n^2 s_1^{1-\f2n}}
\end{align}
and set
\begin{align}\label{deltarestracc}
\delta := \min \bigg\{\f{M_1-M_0}{M_0(s_1-s_*)}e^{\lambda (s_*-s_1)}, \f{\f{m}{\omega_n}-M_1}{M_0} \cdot \f{\kappa_0 e^{\kappa_0 s_1}\cdot e^{\lambda (s_*-s_1)}}{e^{\kappa_0 R^n}-e^{\kappa_0 s_1}}, 1 \bigg\}>0.
\end{align}
Next, we successively write
\begin{align}\label{bandaacc}
b:=\delta s_*^2 \qquad \text{and} \qquad a:=M_0 \cdot \f{s_*+b}{s_*},
\end{align}
and select
\begin{align}\label{candkacc}
c:= \f{ab}{\lambda (s_*+b)^2} e^{-\lambda s_*} \quad \text{as well as} \quad k:= M_0 -ce^{\lambda s_*}.
\end{align}
Now using the fundamental theorem of calculus, \eqref{candkacc} as well as \eqref{bandaacc} and the first restriction in \eqref{deltarestracc}, we may estimate
\begin{align}\label{Wests1acc}
ce^{\lambda s_1} + k &=ce^{\lambda s_*} + k + \int_{s_*}^{s_1} \lambda c e^{\lambda s} ds \notag\\
&\leq ce^{\lambda s_*} + k + (s_1-s_*) \lambda c e^{\lambda s_1} \notag\\
&= M_0 + (s_1-s_*) \f{ab}{(s_*+b)^2} e^{\lambda (s_1-s_*)} \notag\\
&= M_0 + (s_1-s_*)\f{M_0 b}{s_*(s_*+b)} e^{\lambda (s_1-s_*)} \notag\\
&= M_0 + (s_1-s_*)\f{M_0 \delta s_*^2}{s_*^2+\delta s_*^3} e^{\lambda (s_1-s_*)} \notag\\
&\leq M_0 + (s_1-s_*)M_0 \delta e^{\lambda (s_1-s_*)} \notag\\
&\leq M_0 + (s_1-s_*)M_0 \f{M_1-M_0}{M_0(s_1-s_*)}e^{\lambda (s_*-s_1)} e^{\lambda (s_1-s_*)} \notag\\
&= M_0 + M_1 -M_0 \notag\\
&= M_1
\end{align}
and, similarly, by \eqref{candkacc}, \eqref{bandaacc} and the second restriction in \eqref{deltarestracc},
\begin{align}\label{Wsests1acc}
\lambda c e^{\lambda s_1} &= \f{ab}{(s_*+b)^2} e^{\lambda (s_1-s_*)} \notag\\
&= \f{M_0 b}{s_*(s_*+b)} e^{\lambda (s_1-s_*)} \notag\\
&= \f{M_0 \delta s_*^2}{s_*(s_*+\delta s_*^2)} e^{\lambda (s_1-s_*)} \notag\\
&\leq M_0 \delta e^{\lambda (s_1-s_*)} \notag\\
&\leq M_0 \cdot \f{\f{m}{\omega_n}-M_1}{M_0} \cdot \f{\kappa_0 e^{\kappa_0 s_1}\cdot e^{\lambda (s_*-s_1)}}{e^{\kappa_0 R^n}-e^{\kappa_0 s_1}} e^{\lambda (s_1-s_*)} \notag\\
&= \Big(\f{m}{\omega_n}-M_1\Big) \f{\kappa_0 e^{\kappa_0 s_1}}{e^{\kappa_0 R^n}-e^{\kappa_0 s_1}}.
\end{align}
Now \eqref{Wests1acc} and \eqref{Wsests1acc} allow us to obtain
\begin{align*}
\f{\f{m}{\omega_n}-ce^{\lambda s_1}-k}{\lambda c e^{\lambda s_1}} & \geq \f{\f{m}{\omega_n}-M_1}{\lambda c e^{\lambda s_1}} \\
&\geq \f{e^{\kappa_0 R^n}-e^{\kappa_0 s_1}}{\kappa_0 e^{\kappa_0 s_1}},
\end{align*}
whence we may rely on
\[
\f{e^{x R^n}-e^{x s_1}}{x e^{x s_1}} \to \infty \qquad \text{for }x \to \infty
\]
and the intermediate value theorem to assert the existence of $\kappa>0$ such that
\begin{align}\label{kappasubsolacc}
\kappa \geq \kappa_0 \quad \text{and} \quad \f{e^{\kappa R^n}-e^{\kappa s_1}}{\kappa e^{\kappa s_1}}=\f{\f{m}{\omega_n}-ce^{\lambda s_1}-k}{\lambda c e^{\lambda s_1}}.
\end{align}
Thereupon, we finally choose
\begin{align}\label{dsubsolacc}
d:= \f{\lambda c e^{\lambda s_1}}{\kappa e^{\kappa s_1}}
\end{align}
as well as
\begin{align}\label{lsubsolacc}
l:= \f{m}{\omega_n} - de^{\kappa R^n},
\end{align}
and define $\uW= \uW^{(s_0,s_1,M_0,M_1)}$ via
\begin{align*}
\uW(s)=
\left\lbrace
\begin{array}{r@{}l@{\quad}l@{\quad}l@{\,}c}
	&\f{as}{s+b}, \qquad &0\leq s\leq s_*,\\[1mm]
	&ce^{\lambda s} + k, \qquad &s_*< s < s_1,\\[1mm]
	&de^{\kappa s} + l, \qquad &s_1 \leq s \leq R^n.
\end{array}\right.
\end{align*}
We observe that the respective second declarations in \eqref{bandaacc} and \eqref{candkacc} ensure that
\begin{align*}
\uW(s_*)-\lim_{s \searrow s_*} \uW(s) &= \f{as_*}{s_*+b} - (ce^{\lambda s_*} + k) \\
&= M_0 - ce^{\lambda s_*} - M_0 + ce^{\lambda s_*}\\
&= 0,
\end{align*}
whereas by \eqref{dsubsolacc}, \eqref{lsubsolacc} and \eqref{kappasubsolacc}
\begin{align*}
\uW(s_1)-\lim_{s \nearrow s_1} \uW(s) &= de^{\kappa s_1} + l - (ce^{\lambda s_1} + k) \\
&= \f{\lambda c e^{\lambda s_1}}{\kappa e^{\kappa s_1}} e^{\kappa s_1} + \f{m}{\omega_n} - \f{\lambda c e^{\lambda s_1}}{\kappa e^{\kappa s_1}} e^{\kappa R^n} - (ce^{\lambda s_1} + k) \\
&= \f{\lambda c e^{\lambda s_1}}{\kappa e^{\kappa s_1}} (e^{\kappa s_1}-e^{\kappa R^n}) + \f{m}{\omega_n} - (ce^{\lambda s_1} + k) \\
&= \lambda c e^{\lambda s_1}\cdot \f{e^{\kappa s_1}-e^{\kappa R^n}}{\kappa e^{\kappa s_1}} + \f{m}{\omega_n} - ce^{\lambda s_1} - k \\
&= - \Big(\f{m}{\omega_n} - ce^{\lambda s_1} - k\Big) + \f{m}{\omega_n} - ce^{\lambda s_1} - k \\
&= 0.
\end{align*}
Hence $\uW$ is continuous on $[0,R^n]$. Moreover, computing
\begin{align*}
\uW_s(s)=
\left\lbrace
\begin{array}{r@{}l@{\quad}l@{\quad}l@{\,}c}
	&\f{ab}{(s+b)^2}, \qquad &0\leq s < s_*,\\[1mm]
	&\lambda ce^{\lambda s}, \qquad &s_*< s < s_1,\\[1mm]
	&\kappa de^{\kappa s}, \qquad &s_1 < s \leq R^n,
\end{array}\right.
\end{align*}
we may utilize the definition of $c$ in \eqref{candkacc} to obtain
\begin{align*}
\lim_{s\nearrow s_*} \uW_s(s)-\lim_{s \searrow s_*} \uW_s(s) &= \f{ab}{(s_*+b)^2} - \lambda ce^{\lambda s_*}\\
&= \f{ab}{(s_*+b)^2} - \lambda \cdot \f{ab}{\lambda (s_*+b)^2} e^{-\lambda s_*} e^{\lambda s_*}\\
&= 0,
\end{align*}
which together with
\begin{align*}
\lim_{s\nearrow s_1} \uW_s(s)-\lim_{s \searrow s_1} \uW_s(s) &= \lambda ce^{\lambda s_1} - \kappa de^{\kappa s_1}\\
&= \lambda ce^{\lambda s_1} - \kappa \f{\lambda c e^{\lambda s_1}}{\kappa e^{\kappa s_1}}e^{\kappa s_1}\\
&= 0,
\end{align*}
once more established by invoking \eqref{dsubsolacc}, yields the conclusion that actually $\uW \in C^1([0,R^n])$ and $\uW_s>0$ on $[0,R^n]$.\\
Since 
\[
\uW(s_*)=\f{as_*}{s_*+b}=M_0
\]
by \eqref{bandaacc} and $s_*\leq s_0$ by \eqref{s*smallers0acc}, the monotonicity of $\uW$ implies that $\uW(s_0)\geq M_0$. On the other hand, \eqref{Wests1acc} asserts $\uW(s_1)\leq M_1$, and thus \eqref{s0M0s1M1}.\\
We also easily obtain $\uW(0)=0$ and, once more due to \eqref{lsubsolacc},
\[
\uW(R^n)=de^{\kappa R^n} + l = \f{m}{\omega_n}.
\]
As
\begin{align*}
\uW_{ss}(s)=
\left\lbrace
\begin{array}{r@{}l@{\quad}l@{\quad}l@{\,}c}
	&\f{-2ab}{(s+b)^3}, \qquad &0\leq s < s_*,\\[1mm]
	&\lambda^2 ce^{\lambda s}, \qquad &s_*< s < s_1,\\[1mm]
	&\kappa^2 de^{\kappa s}, \qquad &s_1 < s \leq R^n,
\end{array}\right.
\end{align*}
we observe that indeed $\uW_{ss}\geq 0$ in $(s_1,R^n)$, and that additionally, once more falling back on \eqref{deltarestracc} to estimate $b=\delta s_*^2 \leq s_*^2$, we acquire
\begin{align}\label{innerregionacc}
-n^2 s^{2-\f2n}\uW_{ss} - n \uW \uW_s + \mu s \uW_s &= \f{2n^2 ab s^{2-\f2n}}{(s+b)^3} - \f{n a^2 b s}{(s+b)^3}+ \f{\mu ab s}{(s+b)^2} \notag\\
&= \f{nabs}{(s+b)^3} \cdot \Big\{2ns^{1-\f2n}-a+\f{\mu}n (s+b) \Big\} \notag\\
&\leq \f{nabs}{(s+b)^3} \cdot  \Big\{2ns_*^{1-\f2n}-a+\f{\mu}n (s_*+b) \Big\} \notag\\
&= \f{nabs}{(s+b)^3} \cdot  \Big\{2ns_*^{1-\f2n}-M_0 \cdot \f{s_*+b}{s_*}+\f{\mu}n (s_*+b) \Big\} \notag\\
&\leq \f{nabs}{(s+b)^3} \cdot  \Big\{2ns_*^{1-\f2n}-M_0 +\f{\mu}n (s_*+s_*^2) \Big\} \notag\\
&\leq 0 \qquad \text{for }s\in [0,s_*),
\end{align}
where we relied on \eqref{bandaacc} and \eqref{s*smallenoughacc}. In the interior region we draw on our restriction on $\lambda$ in \eqref{lambdas*acc} to see that
\begin{align}\label{interiorregionacc}
-n^2 s^{2-\f2n}\uW_{ss} - n \uW \uW_s + \mu s \uW_s &= -n^2 \lambda^2 c s^{2-\f2n} e^{\lambda s}- n (c e^{\lambda s} + k) c \lambda e^{\lambda s} + \mu \lambda c s e^{\lambda s} \notag\\
&= c\lambda e^{\lambda s} \Big\{ -n^2 \lambda s^{2-\f2n} - n (c e^{\lambda s} + k) + \mu s \Big\} \notag\\
&\leq c\lambda e^{\lambda s} \Big\{ -n^2 \lambda s^{2-\f2n} + \mu s \Big\} \notag\\
&= cs\lambda e^{\lambda s} \Big\{\mu -n^2 \lambda s^{1-\f2n} \Big\} \notag\\
&\leq cs\lambda e^{\lambda s} \Big\{\mu -n^2 \lambda s_*^{1-\f2n} \Big\} \notag\\
&\leq 0 \qquad \text{for }s\in (s_*,s_1),
\end{align}
whereas similarly, in the outer region we resort to $\kappa \geq \kappa_0=\f\mu{n^2 s_1^{1-\f2n}}$ by \eqref{kappasubsolacc} and \eqref{kappa0defacc} to obtain
\begin{align}\label{outerregionacc}
-n^2 s^{2-\f2n}\uW_{ss} - n \uW \uW_s + \mu s \uW_s &= -n^2 \kappa^2 d s^{2-\f2n} e^{\kappa s}- n (d e^{\kappa s} + l) d \kappa e^{\kappa s} + \mu \kappa d s e^{\kappa s} \notag\\
&= d\kappa e^{\kappa s} \Big\{ -n^2 \kappa s^{2-\f2n} - n (d e^{\kappa s} + l) + \mu s \Big\} \notag\\
&\leq d\kappa e^{\kappa s} \Big\{ -n^2 \kappa s^{2-\f2n} + \mu s \Big\} \notag\\
&= ds\kappa e^{\kappa s} \Big\{\mu -n^2 \kappa s^{1-\f2n} \Big\} \notag\\
&\leq ds\kappa e^{\kappa s} \Big\{\mu -n^2 \kappa_0 s_1^{1-\f2n} \Big\} \notag\\
&\leq 0 \qquad \text{for }s\in (s_1,R^n].
\end{align}
By combining \eqref{innerregionacc}, \eqref{interiorregionacc} and \eqref{outerregionacc}, we arrive at \eqref{diffineqexcepts*s1acc} and thereby complete the proof.
\end{bew}

We could apply  this construction directly to prove the monotonicity asserted in Lemma \ref{infmoninclemacc}. However, in order to simultaneously prepare for Lemma \ref{wtomomeganforsgreater0acclemma}, we proceed along a slightly different route.\\
The function $\uW$ has the remarkable property that solutions of \eqref{0wacc} emanating from it never drop below it; in fact, the corresponding minimal nonnegative solution is monotone increasing in time:

\begin{lemma}\label{uWstartingdata}
Let $n\geq 3$, $R>0$, $m>0$ and $\mu := \f{mn}{\omega_n R^n}$, and suppose that $s_0,s_1, M_0$ and $M_1$ are such that
\[
0<s_0<s_1<R^n \quad \text{and} \quad 0<M_0<M_1<\f m{\omega_n},
\]
and with $\uW^{(s_0,s_1,M_0,M_1)} \in W^{2,\infty}((0,R^n))\cap C^\infty([0,R^n]\setminus \{s_*,s_1\})$ for some $s_*\leq s_0$ taken from Lemma \ref{staticsubsolacc}, let $w=\lim_{\eps \searrow 0} \weps$ denote the minimal nonnegative solution of \eqref{0wacc} emanating from $w_0:=\uW^{(s_0,s_1,M_0,M_1)}$ in $(0,R^n)\times(0,\infty)$ obtained in Lemma \ref{createdsol}. Then
\begin{align}\label{wtgeq0acc}
w_t\geq 0 \qquad \text{in }(0,R^n]\times(0,\infty);
\end{align}
in particular,
\begin{equation}\label{wgequWacc}
w(s,t)\geq \uW^{(s_0,s_1,M_0,M_1)}(s) \qquad \text{for all }s\in(0,R^n]~\text{and}~t\geq 0.
\end{equation}
\end{lemma}

\begin{bew}
Due to the regularity properties of $\uW:=\uW^{(s_0,s_1,M_0,M_1)}$, we can fix $\eps_0\in(0,\f{R^n \omega_n}{m})$ such that
\[
\uW_s(s) \leq \f1{\eps_0} \qquad \text{for all }s\in(0,R^n),
\]
and note that this guarantees that for each $\eps\in(0,\eps_0)$, we both have
\[
f_\eps(\uW_s(s)) = \uW(s) \qquad \text{for all }s\in(0,R^n)
\]
as well as
\[
\uW(s)\leq \f{s}\eps \qquad \text{for all }s\in(0,R^n),
\]
because $f_\eps(\xi)=\xi$ for all $\xi \in [0,\f1\eps]$ and $\uW(0)=0$.\\
Therefore, for any such $\eps$ the solution $\weps$ of \eqref{0wepsacc} actually satisfies
\begin{align}\label{wepsuWacc}
\left\lbrace
\begin{array}{r@{}l@{\quad}l@{\quad}l@{\,}c}
	&\wepst = n^2 s^{2-\frac{2}{n}} \wepsss + n\weps f_\eps(\wepss) - \mu s \wepss,
	\qquad & s\in (0,R^n), \ t>0, \\[1mm]
	&\weps(0,t)=0, \quad \weps(R^n,t)=\frac{m}{\omega_n},
	\qquad & t>0, \\[1mm]
	& \weps(s,0)=\uW(s), \qquad
	& s\in (0,R^n),
\end{array}\right.
\end{align}
whereas using Lemma \ref{staticsubsolacc}, we see that for $\uw(s,t):=\uW(s),~s\in[0,R^n],~t\geq 0$, we have
\begin{align*}
\left\lbrace
\begin{array}{r@{}l@{\quad}l@{\quad}l@{\,}c}
	&\uw_t \leq n^2 s^{2-\frac{2}{n}} \uw_{ss} + n\uw f_\eps(\uw_s) - \mu s \uw_s,
	\qquad & s\in (0,R^n), \ t>0, \\[1mm]
	&\uw(0,t)=0, \quad \uw(R^n,t)=\frac{m}{\omega_n},
	\qquad & t>0, \\[1mm]
	& \uw(s,0)=\uW(s), \qquad
	& s\in (0,R^n).
\end{array}\right.
\end{align*}
An application of the comparison principle from Lemma \ref{halfcomparison} thus warrants that $\weps \geq \uw$ in $[0,R^n]\times[0,\infty)$ for all $\eps\in(0,\eps_0)$, which in particular entails that for each fixed $h\in(0,1)$, 
\[
\weps^{(h)}(s,t):=\weps(s,t+h), \qquad (s,t)\in[0,R^n]\times[0,\infty), 
\]
has the property that
\[
\weps^{(h)}(s,0)=\weps(s,h)\geq \uW(s)=\weps(s,0) \qquad \text{for all }s\in[0,R^n]~\text{and}~\eps\in(0,\eps_0).
\]
As both the the parabolic equation and the boundary condition in \eqref{wepsuWacc} are autonomous and hence as well satisfied by $\weps^{(h)}$, we may thus once more employ Lemma \ref{halfcomparison} to infer that
\[
\weps^{(h)}-\weps \geq 0 \qquad \text{in }[0,R^n]\times[0,\infty).
\]
Upon dividing by $h$ and taking $h\searrow 0$, this shows that $\wepst \geq 0$ in $(0,R^n]\times(0,\infty)$, so that \eqref{wtgeq0acc} and hence also \eqref{wgequWacc} result from \eqref{wepswlimitacc}.
\end{bew}

We now compare $w$ to these minimal nonnegative solutions emanating from $\uW=\uW^{(s_0,s_1,M_0,M_1)}$. The resulting statement is also pertinent to Lemma \ref{wtomomeganforsgreater0acclemma} later on.

\begin{lemma}\label{risingsubsolsol}
Let $n\geq 3$, $R>0$, $m>0$, $\mu := \f{mn}{\omega_n R^n}$, assume \eqref{initacc} holds, and suppose that $w\in C^0((0,R^n]\times[0,T))\cap C^{2,1}((0,R^n]\times (0,T))$ is a  nonnegative classical solution of \eqref{0wacc} in $(0,R^n)\times(0,T)$ for some $T\in(0,\infty]$, which is such that $w_s\geq 0$ in $(0,R^n)\times(0,T)$ and that
\begin{align}
\inf_{s\in(0,R^n)} w(s,t_0) > 0 \qquad \text{for some}~t_0\in[0,T).
\end{align}
Then for any choice of $s_0\in(0,R^n)$ and each $\eta>0$, there exists a nonnegative classical solution $\uw\in C^0((0,R^n]\times[t_0,\infty))\cap C^{2,1}((0,R^n]\times(t_0,\infty))$ of the boundary value problem in \eqref{0wacc} which has the properties that $\uw_s \geq 0$ in $(0,R^n)\times(t_0,\infty)$, that
\begin{align}\label{uwtgeq02acc}
\uw_t(s,t)\geq 0 \qquad \text{for all }s\in(0,R^n]~\text{and}~t\geq t_0,
\end{align}
that
\begin{align}\label{uwisasubsolacc}
\uw(s,t)\leq w(s,t) \qquad \text{for all }s\in(0,R^n]~\text{and}~t\in[t_0,T),
\end{align}
and
\begin{align}\label{uws0geqM0acc}
\uw(s_0,t)\geq \inf_{s\in(0,R^n)}w(s,t_0)-\eta \qquad \text{for all }t\geq t_0.
\end{align}
\end{lemma}

\begin{bew}
Abbreviating $M:=\inf_{s\in(0,R^n)}w(s,t_0)$ and assuming without loss of generality that $\eta \in (0,M)$, we let $M_0:=M-\eta \in (0, \f{m}{\omega_n})$ and $M_1:=M-\f\eta 2 \in (M_0,\f m{\omega_n})$, and for $s_0\in(0,R^n)$, we may rely on the continuity of $w_s(\cdot,t_0)$ in $[s_0,R^n]$ to fix $c_1>0$ such that $w_s(\cdot,t_0)\leq c_1$ for all $s\in[s_0,R^n]$. Using that $M_1 < \f m{\omega_n}$, we may thereupon pick $s_1 \in(s_0,R^n)$ suitably close to $R^n$ such that
\begin{align}\label{s1closetoRn2}
R^n-s_1 \leq \f{\f m{\omega_n}-M_1}{c_1},
\end{align}
and let $\uW:=\uW^{(s_0,s_1,M_0,M_1)}$ be as provided by Lemma \ref{staticsubsolacc}.\\
Then since $\uW_s \geq 0$, the second property in \eqref{s0M0s1M1} warrants that for all $s\in(0,s_1]$ we have
\begin{align}\label{uWestsmalls2}
\uW(s) \leq \uW(s_1) \leq M_1 < M = \inf_{\sigma \in (0,R^n)} w(\sigma,t_0) \leq w(s,t_0).
\end{align}
For larger $s$, we firstly make use of our choice of $c_1$ to estimate
\begin{align}\label{westlinearsub}
w(s,t_0)&\geq w(R^n,t_0)-\|w_s(\cdot,t_0)\|_{L^\infty((s_0,R^n))}\cdot (R^n-s) \notag\\
&\geq \f m{\omega_n} - c_1\cdot (R^n-s) \qquad \text{for all }s \in (s_0,R^n],
\end{align}
whereas from the convexity of $\uW$ on $(s_1,R^n]$, as asserted by Lemma \ref{staticsubsolacc}, as well as from \eqref{s1closetoRn2} and \eqref{s0M0s1M1}, we obtain that
\begin{align*}
\uW(s) &\leq \uW(R^n) - \f{\uW(R^n)-\uW(s_1)}{R^n-s_1}\cdot(R^n-s)\\
&= \f{m}{\omega_n} - \f{\f{m}{\omega_n}-\uW(s_1)}{R^n-s_1}\cdot(R^n-s)\\
&\leq \f{m}{\omega_n} - \f{\f{m}{\omega_n}-M_1}{R^n-s_1}\cdot(R^n-s)\\
&\leq \f{m}{\omega_n} - c_1\cdot(R^n-s) \qquad \text{for all }s\in(s_1,R^n].
\end{align*}
Together with \eqref{uWestsmalls2} and \eqref{westlinearsub}, this shows that $\uW(s)\leq w(s,t_0)$ for all $s\in(0,R^n]$, so that the comparison principle from Lemma \ref{halfcomparison} asserts that if in accordance with Lemma \ref{createdsol} we let
\[
\uw(s,t):=\lim_{\eps \to 0} \weps (s,t-t_0), \qquad (s,t)\in(0,R^n]\times[t_0,\infty),
\]
with $\weps$ denoting solutions of \eqref{0wepsacc} emanating from $\weps(s,0):=\min \{\f s\eps, \uW(s)\},~s\in[0,R^n],~\eps \in(0,\f{\omega_n R^n}{m})$, then indeed \eqref{uwisasubsolacc} is valid.\\
Due to our selection of $\uW$, \eqref{wtgeq0acc} precisely states that $\uw$ moreover satisfies \eqref{uwtgeq02acc}, while the monotonicity property $\uw_s\geq 0$ is obvious by Lemma \ref{createdsol}.\\
Finally, \eqref{wgequWacc} along with the left inequality in \eqref{s0M0s1M1} and our definition of $M_0$ ensures that
\[
\uw(s_0,t)\geq \uW(s_0) \geq M_0 = M-\eta
\]
and that hence also \eqref{uws0geqM0acc} holds.
\end{bew}

This shows that $w$ can be bounded from below by its infimum at a given time, thereby confirming the desired monotonicity property.

\begin{lemma}\label{infmoninclemacc}
Let $n\geq 3$, $R>0$, $m>0$, $\mu := \f{mn}{\omega_n R^n}$, and suppose that $w_0$ satisfies \eqref{initacc}. Furthermore, assume that for some $T>0$, $w\in C^0((0,R^n]\times[0,T))\cap C^{2,1}((0,R^n)\times(0,T))$ is a nonnegative classical solution of \eqref{0wacc} in $(0,R^n)\times(0,T)$ such that $w_s\geq 0$ in $(0,R^n)\times(0,T)$. Then
\begin{align}\label{massmonotonacc}
w(s,t) \geq \inf_{\sigma\in(0,R^n)} w(\sigma,t_0) \qquad \text{for all }t_0\in[0,T)~\text{and any}~(s,t)\in(0,R^n]\times [t_0,T).
\end{align}
\end{lemma}

\begin{bew}
It is clearly sufficient to restrict ourselves to considering the case when $\inf_{\sigma\in(0,R^n)} w(\sigma,t_0)$ is positive. We fix $s_0\in(0,R^n)$ and $\eta>0$, and thereupon invoke Lemma \ref{risingsubsolsol} to find $\uw\in C^0((0,R^n]\times[t_0,\infty))\cap C^{2,1}((0,R^n]\times(t_0,\infty))$ satisfying \eqref{uwisasubsolacc} and \eqref{uws0geqM0acc}. These two properties yield
\[
w(s_0,t) \geq \uw(s_0,t) \geq \inf_{s\in(0,R^n)}w(s,t_0)-\eta \qquad \text{for all }t\in[t_0,T),
\]
and thereby entail \eqref{massmonotonacc} due to the fact that $\eta>0$ and $s_0\in(0,R^n)$ had been chosen arbitrarily.

\end{bew}

\subsection{Strict monotonicity}

Building upon the established monotonicity, we now go on to show that $w(0,\cdot)$ is in fact strictly increasing.
To that end, we shall construct a suitable subsolution. As preparation, by means of elementary calculus we first establish

\begin{lemma}\label{uvpacclemma}
Let $\eta>0$, $a\in\R$, $b>a$, $a_0 \in \R$ and $\vp \in C^0([a,b])\cap C^1((a,b])$ be such that $\vp(a)=a_0$ and $\vp'(\xi)>0$ for all $\xi\in(a,b]$.\\
Then there exists $\uvp \in C^1([a,b])\cap C^2((a,b])$ such that $\uvp(a)=a_0$, $a_0<\uvp(\xi)<\vp(\xi)$ for all $\xi\in(a,b]$ and $\uvp'(\xi)>0$ as well as $\uvp''(\xi)>0$ for all $\xi\in(a,b]$.\\
Furthermore, this function satisfies $\uvp(b) \leq a_0 + \eta$.
\end{lemma}

\begin{bew}
Fixing a strictly decreasing sequence $(\xi_j)_{j\in\N}\subset(a,b]$ such that $\xi_1=b$ and $\xi_j \searrow a$ as $j\to\infty$, we recursively define
\begin{align}\label{eta1defacc}
\eta_1:=\min \Big\{\f{\vp'(b)}2, \f \eta{b-a} \Big\}
\end{align}
and
\begin{align}\label{etajdefacc}
\eta_j := \min \Big\{\f12 \eta_{j-1}, \min_{\xi \in [\xi_j,b]} \f{\vp'(\xi)}2 \Big\} \qquad \text{for }j\geq 2.
\end{align}
Then by positivity of $\vp'$ on $(a,b]$, it follows that for all $j \geq 1$ we have $\eta_j>0,~\eta_{j+1}<\eta_j$ and $\vp'(\xi) > \eta_j$ for all $\xi\in[\xi_j,b]$, and that moreover $\eta_j \searrow 0$ as $j\to \infty$.\\
For our further procedure, we note that for any choice of real numbers $A,B,\ueta$ and $\oeta$ fulfilling $A<B$ and $\ueta < \oeta$, by means of an elementary construction it is possible to find $\chi \equiv \chi_{(A,B,\ueta,\oeta)} \in C^1([A,B])$ such that $\chi(A)=\ueta$, $\chi(B)=\oeta$, $\chi'(A)=\chi'(B)=1$ and $\chi'>0$ on $[A,B]$.\\
Utilizing this four-parameter formulation, we introduce a function $\phi$ on $(a,b]$ by letting
\begin{align}\label{phidefuvpacc}
\phi(\xi):=\chi_{(\xi_{j+1},\xi_j,\eta_{j+2},\eta_{j+1})}(\xi) \qquad \text{if }\eta\in(\xi_{j+1},\xi_j]~\text{for some }j\geq 1.
\end{align}
It can easily be verified that through this indeed a function $\phi\in C^1((a,b])$ is well-defined, and that both $\phi$ and $\phi'$ are positive on $(a,b]$ with $\phi(\xi) \searrow 0$ as $\xi \searrow a$, so that setting $\phi(a):=0$ we can achieve that moreover $\phi\in C^0([a,b])$.\\
Apart from that, whenever $j\in\N$ and $\xi\in (\xi_{j+1},\xi_j]$, according to the monotonicity of $\chi_{(\xi_{j+1},\xi_j,\eta_{j+2},\eta_{j+1})}$ and our choice of $\eta_{j+1}$ we have
\[
\phi(\xi)=\chi_{(\xi_{j+1},\xi_j,\eta_{j+2},\eta_{j+1})}(\xi) \leq \chi_{(\xi_{j+1},\xi_j,\eta_{j+2},\eta_{j+1})}(\xi_j)=\eta_{j+1} \leq \min_{\tilde \xi \in [\xi_{j+1},b]} \f{\vp'(\tilde \xi)}2 < \vp'(\xi)
\]
and hence $\phi < \vp'$ on $(a,b]$.\\
Therefore,
\begin{align}\label{uvpdefacca0}
\uvp(\xi):= a_0 + \int_a^\xi \phi(\tilde \xi) d \tilde \xi, \qquad \xi\in[a,b],
\end{align}
belongs to $C^1([a,b])\cap C^2((a,b])$ with $\uvp(a)=a_0$ and clearly satisfies $\uvp < \vp$ on $(a,b]$ as well as $\uvp''=\phi'>0$ throughout $(a,b]$. Lastly, we may verify that by \eqref{uvpdefacca0}, \eqref{phidefuvpacc}, \eqref{etajdefacc} and \eqref{eta1defacc},
\begin{align*}
\uvp(b) &= a_0 + \int_a^b \phi(\tilde \xi) d \tilde \xi \\
&\leq a_0 + \int_a^b \eta_1 d \tilde \xi \\
&= a_0 + (b-a) \eta_1 \\
&\leq a_0 + \eta,
\end{align*}
thus concluding the proof.
\end{bew}

We can suitably encapsulate transport to create a function which lies at the heart of not only Lemma \ref{strictmonotonacclemma} but also Lemma \ref{massaccumulatestotallyacc}. Great care is needed in the construction to ensure that \eqref{psidiffineqacc} still holds near points where $\phi$ enters into zero.

\begin{lemma}\label{functionpsidiffineqpluginfct}
Let $n\geq 3$, $R>0$, $\mu>0$, $a_0>0$, $t_0>0$ and $\tau>0$. Then there exist $\tau_0 \in (0,\tau]$ and $s_*=s_*(n,\mu,a_0,\tau) \in(0,\f{R^n}2)$ such that for any $s_0 \in (0,s_*]$, there are discrete sets $N(t)\subset [0,s_0]$ for each $t\in[t_0,t_0+\tau_0)$ and a function $\psi \in C^1([0,s_0]\times[t_0,t_0+\tau_0))\cap C^2(\mathcal G)$, where
\begin{align}\label{calGdefacc}
\mathcal G:=\{(s,t)\in[0,s_0]\times[t_0,t_0+\tau_0):~s \notin N(t)\},
\end{align}
such that $\psi$ has the properties that 
\begin{align}\label{psisislocallyboundedacc}
0 \leq \psi_s \in L^\infty_{loc}([0,s_0]\times[t_0,t_0+\tau_0)),
\end{align}
that 
\begin{align}\label{psi0at0leqs0ats0acc}
\psi(0,t)=0 \quad \text{and }\quad \psi(s_0,t)\leq s_0 \qquad \text{for all }t\in[t_0,t_0+\tau_0),
\end{align}
that furthermore
\begin{align}\label{psiissmallersinitiallyacc}
\psi(s,0) \leq s \qquad \text{for all }s\in[0,s_0],
\end{align}
that the differential inequality
\begin{align}\label{psidiffineqacc}
\psi_t - n^2 s^{2-\f2n} \psi_{ss} - (n a_0 - \mu s)\psi_s \leq 0
\end{align}
holds in $\mathcal G$, and that moreover for each $s\in(0,s_0]$,
\begin{align}\label{limpsigeqs02acc}
\lim_{t \nearrow t_0 + \tau_0} \psi(s,t) \geq \f{s_0}2.
\end{align}
\end{lemma}

\begin{bew}
We first set
\begin{align}\label{Adefstrictmonincacc}
A:= \f{n a_0}4
\end{align}
and, abbreviating $\alpha:=2-\f2n$,
\begin{align}\label{s0defstrictmonincacc}
s_*:=\min\Big\{\f{R^n}2, \f{na_0}{4\mu}, \Big(\f{a_0}{8n}\Big)^\f1\alpha, \Big(\f{a_0}{8n}\Big)^\f1{\alpha-1}, 1, \f{A^2\tau^2}4\Big\}.
\end{align}

Choosing any $s_0 \in (0,s_*]$, we then define $\phi \in C^1([0,\infty))\cap C^2([0,\infty)\setminus \{y \in [0,\infty)~\vert~y=2\sqrt{s_0}\})$ via
\begin{align*}
\phi(y):= \begin{cases}
(2\sqrt{s_0}-y)^2, \qquad &\text{if }2\sqrt{s_0} \geq y,\\
0, \qquad &\text{otherwise},
\end{cases}
\end{align*}
and thereupon construct $\psi \in C^1([0,s_0]\times[t_0,t_0+\tau_0))\cap C^2(\mathcal G)$ as
\begin{align*}
\psi(s,t):=\f{s_0}2 \cdot \f{s}{s+\phi(s+A(t-t_0))}+\f s2, \qquad (s,t)\in [0,s_0]\times[t_0,t_0+\tau_0),
\end{align*}
with $\tau_0=\f{2\sqrt{s_0}}A$ and $N(t)=\{s \in [0,s_0]~\vert~s=2\sqrt{s_0}-A(t-t_0)\}$ for $t\in[t_0,t_0+\tau_0)$, and $\mathcal G$ as in \eqref{calGdefacc}.\\
Next, we calculate
\begin{align}\label{psitsstrictmonincacc}
\psi_t (s,t) = \f{s_0}2 \cdot \f{-As \phi'}{(s+\phi)^2} \quad \text{as well as } \quad \psi_s = \f{s_0}2\cdot \f{\phi - s \phi'}{(s+\phi)^2}+\f12 \qquad s\in[0,s_0],~t\in[t_0,t_0+\tau_0),
\end{align}
as well as
\begin{align*}
\psi_{ss}= \f{s_0}2 \Big(\f{-s\phi''}{(s+\phi)^2} - \f{2(\phi-s\phi')(1+\phi')}{(s+\phi)^3}\Big), \qquad (s,t)\in \mathcal G,
\end{align*}
where $\phi$ and its derivatives are always to be evaluated at $s+A(t-t_0)$.\\
As furthermore,
\begin{align*}
\phi'(y)= \begin{cases}
-2(2\sqrt{s_0}-y), \qquad &\text{if }2\sqrt{s_0} \geq y,\\
0, \qquad &\text{otherwise},
\end{cases}
\end{align*}
and
\begin{align*}
\phi''(y)= \begin{cases}
2, \qquad &\text{if }2\sqrt{s_0} > y,\\
0, \qquad &\text{if }2\sqrt{s_0} < y,
\end{cases}
\end{align*}
we may observe that $\phi'$ is nonpositive and thus $\psi_s$ is nonnegative.\\
In order to obtain the differential inequality \eqref{psidiffineqacc}, due to the second restriction in \eqref{s0defstrictmonincacc}, we may estimate
\begin{align}\label{psisfactorestacc}
n a_0 - \mu s \geq \f34 n a_0 \qquad \text{for all }s\in[0,s_0],
\end{align}
and further utilizing \eqref{psitsstrictmonincacc}, the fact that $\phi'\leq 0$, and \eqref{Adefstrictmonincacc}, we receive the inequality
\begin{align}\label{psitleqpsisacc}
\psi_t &= \f{s_0}2 \cdot \f{-As \phi'}{(s+\phi)^2} \notag\\
&\leq \f{n a_0}4\cdot \f{s_0}2\cdot\f{-s \phi'}{(s+\phi)^2} \notag\\
&\leq \f{n a_0}4 \psi_s \qquad \text{in }[0,s_0]\times[t_0,t_0+\tau_0).
\end{align}
Moreover, since the third and fourth restriction, respectively, in \eqref{s0defstrictmonincacc} guarantee that
\begin{align*}
n^2 s^\alpha \cdot \f{s_0}2 \cdot \f{s \phi''}{(s+\phi)^2} &\leq \f{n^2 s^{\alpha+1} \cdot s_0}{(s+\phi)^2}\\
&\leq n^2 s^{\alpha-1} \cdot s_0\\
&\leq n^2 s_0^\alpha \\
&\leq \f{na_0}8 \\
&\leq \f{na_0}4 \psi_s \qquad \text{for all }t\in[t_0,t_0+\tau_0)~\text{and}~s\in [0,s_0]\setminus N(t)
\end{align*}
and, again recalling that $\phi'$ is nonpositive and that $\alpha-1=1-\f2n >0$ due to $n\geq 3$,
\begin{align*}
n^2 s^\alpha s_0 \cdot \f{(\phi-s\phi')(1+\phi')}{(s+\phi)^3} &\leq n^2 s^{\alpha-1} s_0 \cdot \f{\phi-s\phi'}{(s+\phi)^2}\\
&\leq 2n^2 s_0^{\alpha-1} \cdot \f{s_0}2 \cdot \f{\phi-s\phi'}{(s+\phi)^2}\\
&\leq \f{na_0}4 \cdot \f{s_0}2 \cdot \f{\phi-s\phi'}{(s+\phi)^2}\\
&\leq \f{na_0}4 \psi_s \qquad \text{for all }t\in[t_0,t_0+\tau_0)~\text{and}~s\in [0,s_0].
\end{align*}
Consequently, we obtain
\begin{align}\label{psissleqpsisacc}
-n^2 s^\alpha \psi_{ss} &= n^2 s^\alpha \cdot \f{s_0}2 \cdot \f{s \phi''}{(s+\phi)^2} + n^2 s^\alpha s_0 \cdot \f{(\phi-s\phi')(1+\phi')}{(s+\phi)^3} \notag\\
& \leq \f{na_0}2 \psi_s \qquad \text{for all }t\in[t_0,t_0+\tau_0)~\text{and}~s\in [0,s_0]\setminus N(t),
\end{align}
whence combining \eqref{psisfactorestacc}, \eqref{psitleqpsisacc} and \eqref{psissleqpsisacc}, we receive
\begin{align*}
\psi_t - n^2 s^\alpha \psi_{ss} - (n a_0 - \mu s)\psi_s &\leq \f{n a_0}4 \psi_s + \f{n a_0}2 \psi_s - \f{3n a_0}4 \psi_s  \notag\\
&= 0 \qquad \text{for all }t\in[t_0,t_0+\tau_0)~\text{and}~s\in [0,s_0]\setminus N(t),
\end{align*}
thus establishing \eqref{psidiffineqacc}. Now for any $\tau_1 \in(0,\tau_0)$, we have
\begin{align*}
 \psi_s &= \f{s_0}2\cdot \f{\phi - s \phi'}{(s+\phi)^2}+\f12\\
&= \f{s_0}2\cdot \f{\phi - s \phi'}{(s+((2\sqrt{s_0}-s-A(t-t_0))_+)^2)^2}+\f12\\
&\leq \f{s_0}2\cdot \f{\phi - s \phi'}{(s+((2\sqrt{s_0}-s-A\tau_1)_+)^2)^2}+\f12\\
&= \f{s_0}2\cdot \f{\phi - s \phi'}{(s+((A(\tau_0-\tau_1)-s)_+)^2)^2}+\f12\\
&\leq \f{s_0}2\cdot \f{\phi - s \phi'}{c_0^2}+\f12 \qquad \text{in }[0,s_0]\times[t_0,t_0+\tau_1]
\end{align*}
with $c_0:=\min\{A^2(\tau_0-\tau_1)^2, A(\tau_0-\tau_1) \}$, where the last inequality has been established by means of elementary calculus. Therefore, $\psi_s$ is bounded locally in $[0,s_0]\times[0,\tau_0)$, validating \eqref{psisislocallyboundedacc}. Similarly, we verify
\begin{align*}
\psi(0,t)=0 \qquad \text{for all }t\in[t_0,t_0+\tau_0),
\end{align*}
and as a result of the nonnegativity of $\phi$, we may infer that
\begin{align*}
\psi(s_0,t) &= \f{s_0}2 \cdot \f{s_0}{s_0+\phi(s+A(t-t_0))}+\f{s_0}2\notag\\
&\leq s_0 \qquad \text{for all }t\in[t_0,t_0+\tau_0),
\end{align*}
confirming \eqref{psi0at0leqs0ats0acc}.
Moreover, in light of $s_0\leq 1$, it follows that
\begin{align*}
\psi(s,t_0) &=\f{s_0}2 \cdot \f{s}{s+(2\sqrt{s_0}-s)^2}+\f{s}2 \notag\\
&\leq \f{s_0}2 \cdot \f{s}{s+s_0}+\f{s}2 \notag\\
&\leq s \qquad \text{for all }s\in(0,s_0],
\end{align*}
that is \eqref{psiissmallersinitiallyacc}.
Lastly, in view of 
\[
\phi(s+A(t-t_0)) \searrow 0 \qquad \text{as }t\nearrow t_0+\tau_0~~ \text{for each }s\in(0,s_0],
\]
we may easily deduce that
\begin{align*}
\lim_{t \nearrow t_0 + \tau_0} \psi(s,t) &= \lim_{t \nearrow t_0 + \tau_0} \f{s_0}2 \cdot \f{s}{s+\phi(s+A(t-t_0))}+\f s2 \\
&= \f{s_0}2 + \f s2 \\
&\geq \f{s_0}2 \qquad \text{for all }s\in(0,s_0]
\end{align*}
and thus \eqref{limpsigeqs02acc} holds, whereas the last restriction in \eqref{s0defstrictmonincacc} asserts that
\[
\tau_0 = \f{2\sqrt{s_0}}A \leq \f{2\sqrt{\f{A^2\tau^2}4}}A = \tau.
\]
\end{bew}

Now we are in position to prove strict monotonicity of $w$. To this effect, for $t_0\geq 0$ we devise a subsolution $\uw$ in $[0,s_0]\times[t_0,t_0+\tau_0)$ for some $s_0\in(0,R^n)$ and $\tau_0>0$ arbitrarily small with the property that $\uw(s,t)$ asymptotically grows beyond $w(0,t_0)$ for all $s>0$ as $t \nearrow t_0+\tau_0$.

\begin{lemma}\label{strictmonotonacclemma}
Let $n\geq 3$, $m>0$, $\mu := \f{mn}{\omega_n R^n}$ and $R>0$, and suppose $w_0$ satisfies \eqref{initacc}. Moreover, assume that for some $T>0$, $w\in C^0((0,R^n]\times[0,T))\cap C^{2,1}((0,R^n)\times(0,T))$ is a nonnegative classical solution of \eqref{0wacc} with $w_s\geq 0$ in $(0,R^n)\times(0,T)$, and suppose that
\begin{align*}
\inf_{s\in(0,R^n)} w(s,t_0) > 0 \qquad \text{for some }t_0\in(0,T).
\end{align*}
Then
\begin{align}\label{massstrictlymonotonacc}
\inf_{s\in(0,R^n)} w(s,t) > \inf_{s\in(0,R^n)} w(s,t_0) \qquad \text{for all }t\in(t_0,T).
\end{align}
\end{lemma}

\begin{bew}
By contradiction, we assume that \eqref{massstrictlymonotonacc} was false and thus, due to the monotonicity ensured by Lemma \ref{infmoninclemacc}, there existed $\tau>0$ such that, abbreviating $a_0:=\inf_{s\in(0,R^n)} w(s,t_0)>0$,
\begin{align}\label{infstaysconstantfortauacc}
\inf_{s\in(0,R^n)} w(s,t) = a_0 \qquad \text{for all }t\in[t_0,t_0+\tau].
\end{align}
We may then invoke Lemma \ref{functionpsidiffineqpluginfct} to find $\tau_0 \in (0,\tau]$, $s_0=s_0(n,\mu,a_0,\tau) \in(0,\f{R^n}2)$, $\mathcal G$ as in \eqref{calGdefacc} and a function $\psi \in C^1([0,s_0]\times[t_0,t_0+\tau_0))\cap C^2(\mathcal G)$ satisfying \eqref{psisislocallyboundedacc} -- \eqref{limpsigeqs02acc}.\\
Now since $w_s>0$ in $(0,R^n)\times[t_0,t_0+\tau]$ by \eqref{wsrealgreaterzeroacc} and $w(s_0, \cdot)$ is continuous in $[t_0,t_0+\tau]$, there exists $\eta>0$ with the property that
\begin{align}\label{ws0tlowerestacc}
w(s_0, t) \geq a_0 + \eta \qquad \text{for all }\eta \in [t_0,t_0+\tau].
\end{align}
Then, again utilizing that $w_s(\cdot,t_0)>0$ in $(0,s_0]$ as $s_0 \leq \f{R^n}2$, by Lemma \ref{uvpacclemma} there exists a function $\uvp \in C^1([0,s_0])\cap C^2((0,s_0])$ with the properties that 
\begin{align}\label{uvpspatialboundariesacc}
\uvp(0)=a_0 \quad \text{and} \quad \uvp(s_0) \leq a_0+\eta,
\end{align}
as well as
\begin{align}\label{uvplesswt0acc}
a_0<\uvp(\xi)<w(\xi,t_0) \qquad \text{for all }\xi\in(0,s_0]
\end{align}
and 
\begin{align}\label{uvpprime12greater0acc}
\uvp'(\xi)>0 \quad \text{and } \quad \uvp''(\xi)>0 \qquad \text{for all }\xi\in(0,s_0].
\end{align}
In turn, we establish
\begin{align*}
\uw(s,t):= \uvp(\psi(s,t)), \qquad (s,t)\in [0,s_0]\times[t_0,t_0+\tau_0),
\end{align*}
and observe that accordingly, $\uw$ lies in $C^1([0,s_0]\times[t_0,t_0+\tau_0))\cap C^2(\mathcal G)$ and, as $\psi_s\geq 0$ by \eqref{psisislocallyboundedacc}, satisfies
\begin{align}\label{psidiffineqfinalacc}
\uw_t - n^2s^\alpha \uw_{ss} - n \uw \uw_s + \mu s \uw_s &= \uvp' \cdot \psi_t - n^2 s^\alpha(\uvp'' \cdot \psi_s^2 + \uvp' \cdot \psi_{ss}) - n \uvp\cdot\uvp'\cdot \psi_s + \mu s \uvp' \cdot \psi_s \notag\\
&\leq \uvp' \cdot \Big\{ \psi_t - n^2 s^\alpha \psi_{ss} - n \uvp \psi_s + \mu s\psi_s \Big\} \notag\\
&\leq \uvp' \cdot \Big\{ \psi_t - n^2 s^\alpha \psi_{ss} - n a_0 \psi_s + \mu s\psi_s \Big\} \notag\\
&= \uvp' \cdot \Big\{ \psi_t - n^2 s^\alpha \psi_{ss} - (n a_0 - \mu s)\psi_s \Big\} \notag\\
&\leq 0
\end{align}
in $\mathcal G$ on account of \eqref{uvpprime12greater0acc}, \eqref{uvpspatialboundariesacc} and \eqref{psidiffineqacc}.\\
Once more utilizing \eqref{psisislocallyboundedacc}, $\uw_s = \uvp' \cdot \psi_s$ is bounded locally in $[0,s_0]\times[t_0,t_0+\tau_0)$.\\
By \eqref{psi0at0leqs0ats0acc} and \eqref{uvpspatialboundariesacc}, we verify
\begin{align}\label{uwisa0at0acc}
\uw(0,t)=\uvp(\psi(0,t))=\uvp(0)=a_0 \qquad \text{for all }t\in[t_0,t_0+\tau_0),
\end{align}
and, additionally resorting to the monotonicity of $\uvp$ and \eqref{ws0tlowerestacc}, we may infer that
\begin{align}\label{uwissmallerats0acc}
\uw(s_0,t) &= \uvp(\psi(s_0,t)) \notag\\
&\leq \uvp(s_0) \notag\\
&\leq a_0 + \eta \notag\\
&\leq w(s_0,t) \qquad \text{for all }t\in[t_0,t_0+\tau_0).
\end{align}
Moreover, again in light of the monotonicity of $\uvp$, \eqref{psiissmallersinitiallyacc} and \eqref{uvplesswt0acc}, it follows that
\begin{align}\label{uwissmallerinitiallyacc}
\uw(s,t_0)&=\uvp(\psi(s,t_0)) \notag\\
&\leq \uvp(s) \notag\\
&< w(s,t_0) \qquad \text{for all }s\in(0,s_0].
\end{align}
Given \eqref{psidiffineqfinalacc}, \eqref{uwisa0at0acc}, \eqref{uwissmallerats0acc}, \eqref{uwissmallerinitiallyacc}, the fact that $\uw_s$ is bounded locally in time, and the discreteness of the set $N(t)$ for each $t\in[t_0,t_0+\tau_0)$ in the definition of $\mathcal G$, we may invoke the comparison principle Lemma \ref{halfcomparison} to deduce that
\begin{align*}
\uw(s,t) \leq w(s,t) \qquad \text{for all }s\in(0,s_0]~\text{and}~t\in[t_0,t_0+\tau_0).
\end{align*}
In conjunction with the continuity of $w(s,\cdot)$ for each $s\in(0,s_0]$ and \eqref{limpsigeqs02acc}, this entails that
\begin{align*}
w(s,t_0+\tau_0) &= \lim_{t \nearrow t_0+\tau_0} w(s,t) \\
&\geq \lim_{t \nearrow t_0+\tau_0} \uw(s,t) \\
&= \lim_{t \nearrow t_0+\tau_0} \uvp(\psi(s,t)) \\
&= \uvp\Big(\lim_{t \nearrow t_0+\tau_0}\psi(s,t)\Big) \\
&\geq \uvp\Big(\f{s_0}2\Big) \qquad \text{for all }s\in(0,s_0].
\end{align*}
As $\uvp(\f{s_0}2)>a_0$ by \eqref{uvplesswt0acc}, this implies that
\begin{align*}
w(s,t_0+\tau_0) \geq \uvp\Big(\f{s_0}2\Big) > a_0 \qquad \text{for all }s\in(0,s_0],
\end{align*}
which in view of $\tau_0 \leq \tau$ contradicts \eqref{infstaysconstantfortauacc} and hence confirms that \eqref{massstrictlymonotonacc} must hold.
\end{bew}

\subsection{Right-continuity}

By applying Lemma \ref{uvpacclemma} to the inverse map and extending appropriately, we obtain a bounded concave, strictly increasing function which dominates a given suitably regular increasing map and coincides with it on the left.

\begin{lemma}\label{ovpacclemma}
Let $a\in\R,~b>a$, $a_0\in\R$ and $\vp \in C^0([a,b]) \cap C^1((a,b])$ be such that $\vp(a)=a_0$ and $\vp'\geq 0$ on $(a,b]$.\\
Then there exists $\ovp \in C^0([a,\infty))\cap C^2((a,\infty))$ such that $\ovp(a)=a_0$ as well as 
\[
\ovp'(\xi)>0~\text{and}~\ovp''(\xi)<0 \qquad \text{for all }\xi\in(a,\infty)
\]
and that furthermore $\ovp$ is bounded and fulfills
\[
\ovp(\xi)>\vp(\xi) \qquad \text{for all }\xi\in(a,b].
\]
\end{lemma}

\begin{bew}
Replacing $\vp(\xi)$ by $\vp(\xi)+\xi$ for $\xi\in[a,b]$ if necessary, we may assume that $\vp'>0$ on $(a,b]$.\\
Then $\vp$ is a bijection from $[a,b]$ onto $[a_0,B]$ with $B:=\vp(b)>a_0$, so that $\rho:=\vp^{-1}$ is a well-defined element of $C^0([a_0,B])\cap C^1((a_0,B])$ satisfying $\rho(a_0)=a$ and $\rho'>0$ on $(a_0,B]$.\\
Therefore, Lemma \ref{uvpacclemma} applies so as to provide $\urho \in C^0([a_0,B])\cap C^2((a_0,B])$ with the properties that $\urho(0)=0$ as well as
\[
\urho(\eta) < \rho(\eta), \quad \urho'(\eta)>0 \quad \text{and} \quad \urho''(\eta)>0 \qquad \text{for all }\eta\in(a_0,B].
\]
Thus, if we firstly define $\ovp$ on $[a,\xi_0]$ with $\xi_0:=\urho(B) < \rho(B) = b$ by letting
\[
\ovp(\xi):=\urho^{-1}(\xi), \qquad \xi \in[a,\xi_0],
\]
then $\ovp$ lies in $C^0([a,\xi_0])\cap C^2((a,\xi_0])$ with
\[
\ovp'(\xi)=\f1{\urho'(\urho^{-1}(\xi))} > 0 \qquad \text{for all }\xi \in (a,\xi_0]
\]
and
\[
\ovp''(\xi)=-\f{\urho''(\urho^{-1}(\xi))}{\urho'^3(\urho^{-1}(\xi))} < 0 \qquad \text{for all }\xi \in (a,\xi_0]
\]
as well as $\ovp(a)=a_0$, and since $\rho > \urho$ on $(a_0,B]$, from the strict monotonicity of $\urho^{-1}$ it moreover follows that
\begin{align*}
\ovp(\xi)=\urho^{-1}(\xi)=\urho^{-1}(\rho(\vp(\xi))) > \urho^{-1}(\urho(\vp(\xi))=\vp(\xi) \qquad \text{for all }\xi\in(a,\xi_0].
\end{align*}
To finally extend $\ovp$ so as to exist on all of $[a,\infty)$, we note that $c_1:=\ovp(\xi_0),~c_2:=\ovp'(\xi_0)$ and $c_3:= - \ovp''(\xi_0)$ are all positive, so that if we let
\[
\ovp(\xi):=c_1 + \f{c_2^2}{c_3}\cdot(1-e^{-\f{c_3}{c_2}(\xi-\xi_0)}) \qquad \text{for }\xi\in(\xi_0,\infty),
\]
then it can readily be verified that in fact $\ovp\in C^0([a,\infty)) \cap C^2((a,\infty))$ is bounded with
\[
\ovp'(\xi)= c_2 e^{-\f{c_3}{c_2}(\xi-\xi_0)} > 0 \qquad \text{for all }\xi\in(\xi_0,\infty)
\]
and
\[
\ovp''(\xi)= -c_3 e^{-\f{c_3}{c_2}(\xi-\xi_0)} < 0 \qquad \text{for all }\xi\in(\xi_0,\infty)
\]
and hence, in particular, also
\[
\ovp(\xi) > \ovp(\xi_0) = \urho^{-1}(\xi_0) = B = \vp(b) \geq \vp(\xi) \qquad \text{for all }\xi\in(\xi_0,b]
\]
again due to the monotonicity of $\vp$.
\end{bew}

We use a comparison argument involving the approximate solutions $\weps$ from Lemma \ref{wepsolacc} and a supersolution crafted by shifting the function from the previous lemma to the left, thereby proving that the minimal solutions created in Lemma \ref{createdsol} are right-continuous.

\begin{lemma}\label{rightcontinuityacc}
Let $n\geq 1$, $R>0$, $m>0$, $\mu := \f{mn}{\omega_n R^n}$, and suppose $w_0$ satisfies \eqref{initacc}. Furthermore, assume that for some $T>0$, $w\in C^0((0,R^n]\times[0,T))\cap C^{2,1}((0,R^n]\times(0,T))$ is the nonnegative classical solution of \eqref{0wacc} in $(0,R^n)\times(0,T)$ from Lemma \ref{createdsol}.\\
Then $t \mapsto \inf_{s\in(0,R^n)} w(s,t)$ is continuous from the right on $[0,T)$, that is, for all $t_0 \in [0,T)$ and any $\eta>0$ one can find $\delta \in (0,T-t_0)$ such that
\begin{align}\label{infimumrisesnottoomuch}
\inf_{s\in(0,R^n)} w(s,t) \leq \inf_{s\in(0,R^n)} w(s,t_0) + \eta \qquad \text{for all }t\in(t_0,t_0+\delta).
\end{align}
\end{lemma}

\begin{bew}
Since Lemma \ref{createdsol} guarantees that for any $t_0 \geq 0$, $w(\cdot,t_0)$ lies in $C^0([0,R^n]) \cap C^1((0,R^n])$ and is monotonically increasing, abbreviating $a_0:=\inf_{s\in(0,R^n)} w(s,t_0)$ we may invoke Lemma \ref{ovpacclemma} to fix $\ovp \in C^0([0,\infty))\cap C^2((0,\infty))$ such that $\ovp(0)=a_0$ as well as 
\begin{align}\label{ovpderivativessign}
\ovp'(\xi)>0~\text{and}~\ovp''(\xi)<0 \qquad \text{for all }\xi\in(0,\infty)
\end{align}
and
\begin{align}\label{ovpgreaterwt0}
\ovp(\xi)>w(\xi,t_0) \qquad \text{for all }\xi\in(0,R^n].
\end{align}
Thereupon, we sequentially choose 
\begin{align}\label{AgreaternovpLinf}
A:=n \|\ovp\|_{L^\infty((0,\infty))},
\end{align}
$\delta>0$ small enough such that
\begin{align}\label{deltaovpa0def}
\ovp(A\delta)-a_0 \leq \eta,
\end{align}
which is possible due to continuity of $\ovp$.\\
We proceed to define $\ow \in C^0([0,R^n]\times[t_0,t_0+\delta])\cap C^2((0,R^n]\times(t_0,t_0+\delta])$ via
\[
\ow(s,t):= \ovp(s+A(t-t_0)), \qquad s\in[0,R^n],~t\in[t_0,t_0+\delta],
\]
which is clearly a nonnegative function. Then for $\eps \in (0,\f{R^n\omega_n}m)$, we denote $\weps$ as the classical solutions to \eqref{0wepsacc} from Lemma \ref{wepsolacc} and observe that
\begin{align}\label{initialwepssmallerowacc}
\weps(s,t_0) \leq w(s,t_0) < \ovp(s) = \ow(s,t_0) \qquad \text{for all }s\in(0,R^n]
\end{align}
and each $\eps \in (0,\f{R^n\omega_n}m)$ by \eqref{pointwisewepsacc} and \eqref{ovpgreaterwt0}, as well as
\begin{align}\label{0Rnwepssmallereqowacc}
\weps(0,t)=0 \quad \text{and} \quad \weps(R^n,t) = w(R^n,t) = w(R^n,t_0) \leq \ovp(R^n) \leq \ovp(R^n+A(t-t_0))=\ow(R^n,t)
\end{align}
for all $t\in[t_0,t_0+\delta]$ and $\eps \in (0,\f{R^n\omega_n}m)$ due to \eqref{0wepsacc}, \eqref{ovpgreaterwt0} and \eqref{ovpderivativessign}. Moreover, once more utilizing \eqref{ovpderivativessign}, we see that $\ow_s \geq 0$ and $\ow_{ss}\leq 0$, and thus by \eqref{AgreaternovpLinf} and the fact that $f_\eps(\xi)\leq \xi$ for all $\xi \geq 0$, we may estimate
\begin{align}\label{diffineqwepsowovpacc}
\ow_t - n^2 s^{2-\f2n}\ow_{ss} - n \ow f_\eps(\ow_s) + \mu s \ow_s &= A \ow_s - n^2 s^{2-\f2n}\ow_{ss} - n \ow f_\eps(\ow_s) + \mu s \ow_s \notag\\
&\geq \Big\{A - n \ow \Big\} \ow_s \notag\\
&\geq \Big\{A - n \|\ovp\|_{L^\infty((0,\infty))}\Big\} \ow_s \notag\\
&= 0 \qquad \text{in}~(0,R^n]\times(t_0,t_0+\delta]~\text{for each}~\eps\in(0,\f{R^n\omega_n}m).
\end{align}
In turn, \eqref{initialwepssmallerowacc}, \eqref{0Rnwepssmallereqowacc} and \eqref{diffineqwepsowovpacc} enable us to apply the comparison principle Lemma \ref{halfcomparison} to infer that
\[
\weps(s,t) \leq \ow(s,t) \qquad \text{for all }(s,t)\in(0,R^n]\times[t_0,t_0+\delta)~\text{and each }\eps\in(0,\f{R^n\omega_n}m).
\]
Consequently, passing to the limit as $\eps \searrow 0$, by Lemma \ref{createdsol} we obtain
\[
w(s,t) \leq \ow(s,t) \qquad \text{for all }(s,t)\in(0,R^n]\times[t_0,t_0+\delta),
\]
which together with the continuity of $w(\cdot,t)$ in $[0,R^n]$ and \eqref{deltaovpa0def} implies that
\begin{align*}
w(0,t) \leq \ow(0,t)&= \ovp(A(t-t_0)) \\
&\leq \ovp(A\delta) \\
&\leq a_0 + \eta \\
&= \inf_{s\in(0,R^n)} w(s,t_0) + \eta \qquad \text{for all }t\in[t_0,t_0+\delta)
\end{align*}
and thus \eqref{infimumrisesnottoomuch} holds.
\end{bew}

\section{Asymptotically complete aggregation of mass in the origin}

As $w(0,\cdot)$ is increasing and bounded, we can already conclude that $w(0,t) \to w_\infty$ for some $w_\infty \in (0,\f{m}{\omega_n}]$ as $t\to\infty$ if $w(0,t_0)>0$ for some $t_0\geq 0$. We shall establish that indeed always $w_\infty=\f{m}{\omega_n}$.\\
Utilizing an implicit logarithmic structure, we first demonstrate that the differential equation in \eqref{0wacc} admits no nonconstant, nondecreasing positive steady states.

\begin{lemma}\label{staticsollemma}
Let $n\geq 3$ and $R>0$, and suppose that with some $\mu>0$, $W \in C^0((0,R^n]) \cap C^2((0,R^n))$ is a nondecreasing function satisfying
\begin{align}\label{StatWacc}
0 = n^2 s^{2-\frac{2}{n}} W_{ss} + nWW_s - \mu s W_s, \qquad  s\in (0,R^n),
\end{align}
as well as
\begin{align}\label{Wpositiveacc}
\inf_{s\in(0,R^n)}W(s) > 0.
\end{align}
Then
\begin{align}\label{staticsolconstant}
W_s(s)=0 \qquad \text{for all }s\in(0,R^n).
\end{align}
\end{lemma}

\begin{bew}
In view of a uniqueness property of initial-value problems for \eqref{StatWacc}, it is clear that if \eqref{staticsolconstant} was false, then actually
\begin{align}\label{Wspositiveacc}
W_s(s)>0 \qquad \text{for all }s\in(0,R^n).
\end{align}
To see that this is incompatible with \eqref{Wpositiveacc}, assuming the latter we abbreviate $c_1:= \inf_{s\in(0,R^n)}W(s)>0$ and may then pick $s_0\in(0,R^n)$ small enough such that
\[
\mu s_0 \leq \f{n c_1}2.
\]
Through \eqref{Wpositiveacc}, this entails that since $W_s \geq 0$, using \eqref{StatWacc} we can estimate
\begin{align*}
n^2 s^{2-\f2n}W_{ss}(s) &= -nW(s)W_s(s) + \mu s W_s(s) \\
&\leq -n c_1 W_s(s) + \mu s_0 W_s(s) \\
&\leq -\f{nc_1}2 W_s(s) \qquad \text{for all }s\in(0,s_0),
\end{align*}
so that
\[
\f{W_{ss}(s)}{W_s(s)} \leq - \f{c_1}{2n} s^{\f2n - 2} \qquad \text{for all }s\in(0,s_0),
\]
which upon integration shows that
\begin{align}\label{Wsexponentialineq}
W_s(s) \geq W_s(s_0) \cdot e^{\f{c_1}{2(n-2)}\left(s^{-\f{n-2}2}-s_0^{-\f{n-2}2}\right)} \qquad \text{for all }s\in(0,s_0),
\end{align}
because $n\geq 3$. Since herein $W_s(s_0)$ is positive by \eqref{Wspositiveacc}, observing that
\[
\int_0^{s_0} e^{\f{c_1}{2(n-2)}s^{-\f{n-2}2}} ds = \infty
\]
again due to the inequality $n\geq 3$, from further integration of \eqref{Wsexponentialineq} we can infer that
\[
W(s) = W(s_0) - \int_s^{s_0} W_s(\sigma) d\sigma \to -\infty \qquad \text{as}~s\searrow 0,
\]
which contradicts \eqref{Wpositiveacc} and thereby shows that actually indeed \eqref{staticsolconstant} must hold.
\end{bew}

Building on Lemma \ref{risingsubsolsol} once more, we deduce that $w$ converges to $\f{m}{\omega_n}$ pointwise on $(0,R^n]$.

\begin{lemma}\label{wtomomeganforsgreater0acclemma}
Let $n\geq 3$, $R>0$, $m>0$ and $\mu := \f{mn}{\omega_n R^n}$, and assume $w_0$ satisfies \eqref{initacc}. Moreover, let $w\in C^0((0,R^n]\times[0,\infty))\cap C^{2,1}((0,R^n]\times(0,\infty))$ be a classical solution of \ref{0wacc} in $(0,R^n)\times(0,\infty)$ fulfilling $w_s\geq 0$ in $(0,R^n)\times(0,\infty)$ as well as
\begin{align}\label{infgreaterthanzeroacc}
\inf_{s\in(0,R^n)} w(s,t_0)>0 \qquad \text{for some }t_0\geq 0.
\end{align}
Then for all $s\in(0,R^n)$,
\begin{align}\label{masstowardsorigin}
w(s,t) \to \f m{\omega_n} \qquad \text{as }t \to \infty.
\end{align}
\end{lemma}

\begin{bew}
Since $w$ is bounded, parabolic standard theory (\cite{Ladyzhenskaja1968}) shows that for each $s_0 \in (0,R^n)$, there exist $\theta \in (0,1)$ and $c_1>0$ such that
\begin{align*}
\|w(\cdot,t)\|_{C^\theta([s_0,R^n])} \leq c_1 \qquad \text{for all }t>1,
\end{align*}
which due to the Arzel\`a-Ascoli theorem implies that $(w(\cdot,t))_{t>1}$ is relatively compact in $C^0_{loc}((0,R^n])$. In view of a standard argument, for the verification of \eqref{masstowardsorigin} it is thus sufficient to show that whenever $(t_k)_{k\in\N}\subset(1,\infty)$ and $W_\infty \in C^0((0,R^n])$ are such that $t_k \to \infty$ and
\begin{align}\label{wtktoWinftyacc}
w(\cdot,t_k) \to W_\infty \qquad \text{in }C_{loc}^0((0,R^n])
\end{align}
as $k\to\infty$, we have $W_\infty \equiv \f m{\omega_n}$.\\
To achieve this, using \eqref{infgreaterthanzeroacc} we fix $c_1>0$ such that $c_1<\inf_{s\in(0,R^n)} w(s,t_0)$, and taking any $(s_j)_{j\in\N}\subset(0,R^n)$ such that $s_j \searrow 0$ as $j\to\infty$, we invoke Lemma \ref{risingsubsolsol} to see that for each $j\in\N$ we can find a classical solution $\uw_j \in C^0((0,R^n]\times[t_0,\infty))\cap C^{2,1}((0,R^n]\times[t_0,\infty))$ such that $\uw_{js} \geq 0$ in $(0,R^n)\times(t_0,\infty)$, that
\begin{align}\label{uwjtgeq0acc}
\uw_{jt}(s,t)\geq 0 \qquad \text{for all }s\in(0,R^n]~\text{and}~t>t_0,
\end{align}
that
\begin{align}\label{uwjleqwacc}
\uw_j(s,t) \leq w(s,t) \qquad \text{for all }s\in(0,R^n]~\text{and}~t\geq t_0
\end{align}
and that
\begin{align}\label{uwsjgeqc1acc}
\uw_j(s_j,t) \geq c_1 \qquad \text{for all }t\geq t_0.
\end{align}
In particular, \eqref{uwjtgeq0acc} ensures that for any such $j$,
\begin{align}\label{Wjdefuwjacc}
W_j(s):=\lim_{t\to\infty} \uw_j(s,t), \qquad s\in(0,R^n],
\end{align}
is a well-defined and clearly nondecreasing function fulfilling
\begin{align}\label{0leqWjleqmomnacc}
0 \leq W_j(s) \leq W_j(R^n) = \f m{\omega_n} \qquad \text{for all }s\in(0,R^n]
\end{align}
and, by \eqref{uwsjgeqc1acc},
\begin{align}\label{Wjsjgeqc1acc}
W_j(s_j)\geq c_1.
\end{align}
Moreover, since again by parabolic standard theory we may conclude from the boundedness of $\uw_j$ that $(\uw_j(\cdot,t))_{t>1}$ is relatively compact even in $C^2_{loc}((0,R^n])$ and that $(\uw_{jt}(\cdot,t))_{t>1}$ is relatively compact in $C^0_{loc}((0,R^n])$, it follows from \eqref{Wjdefuwjacc} and \eqref{uwjtgeq0acc} that
\[
\uw_j(\cdot,t) \to W_j \quad \text{in }C^2_{loc}((0,R^n]) \quad \text{and} \quad \uw_{jt}(\cdot,t)\to 0 \quad \text{in }C^0_{loc}((0,R^n]) \qquad \text{as }t\to\infty,
\]
and that hence taking $t\to\infty$ in \eqref{0wacc} yields the identity
\begin{align}\label{Wjstatdiffeqacc}
0=n^2s^{2-\f2n}W_{jss} + n W_j W_{js} - \mu s W_{js} \qquad \text{in }(0,R^n].
\end{align}
Together with \eqref{0leqWjleqmomnacc} and a straightforward regularity argument, this can really be shown to guarantee that $(W_j)_{j\in \N}$ is bounded in $C^3_{loc}((0,R^n])$ and hence relatively compact in $C^2_{loc}((0,R^n])$, meaning that on passing to a subsequence if necessary we may assume that
\begin{align}\label{WjtoWinC2locacc}
W_j \to W \quad \text{in }C^2_{loc}((0,R^n]) \qquad \text{as }j\to\infty
\end{align}
with some nondecreasing $W\in C^2_{loc}((0,R^n])$. Due to \eqref{0leqWjleqmomnacc} and \eqref{WjtoWinC2locacc}, this limit function satisfies
\begin{align}\label{0leqWleqmomacc}
0 \leq W(s) \leq W(R^n) = \f m{\omega_n},
\end{align}
while by \eqref{Wjstatdiffeqacc} and \eqref{WjtoWinC2locacc}, it is a classical solution of
\begin{align}\label{Wstatdiffeqacc}
0=n^2 s^{2-\f2n}W_{ss} + n W W_s - \mu s W_s \qquad \text{in }(0,R^n].
\end{align}
In addition,
\begin{align}\label{Wgeqc1everywhereacc}
W(s) \geq c_1 \qquad \text{for all }s\in(0,R^n],
\end{align}
because for each $j_0 \in \N$ and $j\in\N$ with $j\geq j_0$, \eqref{Wjsjgeqc1acc} along with the monotonicity of $W_j$ for all $j\in\N$ warrants that actually
\[
W_j(s) \geq W_j(s_j) \geq c_1 \qquad \text{for all }s\in[s_j,R^n] \supset [s_{j_0},R^n],
\]
which on taking $j\to \infty$ yields $W\geq c_1$ on $[s_{j_0},R^n]$ for all $j_0\in \N$ due to \eqref{WjtoWinC2locacc}, and which thereby implies \eqref{Wgeqc1everywhereacc} in the limit $j_0\to\infty$.\\
According to Lemma \ref{staticsollemma}, however, this latter property \eqref{Wgeqc1everywhereacc} when combined with \eqref{0leqWleqmomacc} and \eqref{Wstatdiffeqacc} enforces that
\[
W(s)=\f m{\omega_n} \qquad \text{for all }(0,R^n],
\]
so that since for all $j\in\N$ we know from \eqref{uwjleqwacc}, \eqref{wtktoWinftyacc} and \eqref{Wjdefuwjacc} that
\[
W_\infty(s) \geq W_j(s) \qquad \text{for all }s\in(0,R^n],
\]
and since thus \eqref{WjtoWinC2locacc} entails that
\[
W_\infty(s)\geq W(s) \qquad \text{for all }s\in(0,R^n],
\]
it follows from \eqref{0wacc} and the evident nondecrease of $W_\infty$ on $(0,R^n]$ that indeed
\[
\f m{\omega_n} \leq W_\infty(s) \leq W_\infty(R^n) = \f m{\omega_n},
\]
as desired.
\end{bew}

An iterative application of the argument from Lemma \ref{strictmonotonacclemma} reveals that, asymptotically, the entire mass not only concentrates near but enters into the origin.

\begin{lemma}\label{massaccumulatestotallyacc}
Let $n\geq 3$, $R>0$, $m>0$ and $\mu := \f{mn}{\omega_n R^n}$, suppose $w_0$ satisfies \eqref{initacc}, and assume that $w\in C^0((0,R^n]\times[0,\infty))\cap C^{2,1}((0,R^n)\times(0,\infty))$ is a nonnegative classical solution of \eqref{0wacc} with $w_s\geq 0$ in $(0,R^n)\times(0,\infty)$, and presume that
\begin{align*}
\inf_{s\in(0,R^n)} w(s,\bar t) > 0 \qquad \text{for some }\bar t>0.
\end{align*}
Then
\begin{align}\label{massasymptoticallyaggregatedacc}
\inf_{s\in(0,R^n)} w(s,t) \to \f{m}{\omega_n} \qquad \text{as }t\to\infty.
\end{align}
\end{lemma}

\begin{bew}
Writing 
\begin{align}\label{a0defindiraccompleteacc}
a_0:= \inf_{s\in(0,R^n)} w(s,\bar t), 
\end{align}
For any $t_* \geq 0$, Lemma \ref{functionpsidiffineqpluginfct} guarantees the existence of $\tau_0 \in (0,1]$, $s_0=s_0(n,\mu,a_0) \in(0,\f{R^n}2)$ and some discrete set $N(t)\subset [0,s_0]$ for each $t\in[t_*,t_*+\tau_0)$ such that for 
\[
\mathcal G_{t_*}:=\{(s,t)\in[0,s_0]\times[t_*,t_*+\tau_0):~s \notin N(t)\},
\]
there is a function $\psi_{t_*} \in C^1([0,s_0]\times[t_*,t_*+\tau_0))\cap C^2(\mathcal G_{t_*})$ satisfying \eqref{psisislocallyboundedacc} -- \eqref{limpsigeqs02acc}.\\
We may then fix any $\eps\in(0,\f{m}{\omega_n}-a_0)$ and draw upon Lemma \ref{wtomomeganforsgreater0acclemma} to pick $t_0 \geq \bar t$ such that
\begin{align}\label{ws04greatenoughacc}
w\Big(\f{s_0}4,t\Big) \geq \f{m}{\omega_n}-\eps \qquad \text{for all }t\geq t_0.
\end{align}
Furthermore, we define a family of functions $(\uvp_a)_{a\in(0,\f{m}{\omega_n}-\eps)} \subset C^2([0,s_0])$ via
\begin{align}\label{uvpadefacccomplete}
\uvp_a (s) = \begin{cases}
a, \qquad &0\leq s \leq \f{s_0}4,\\
a + (\f4{3s_0})^3(\f{m}{\omega_n}-\eps-a)(s-\f{s_0}4)^3, \qquad & \f{s_0}4 < s \leq s_0,
\end{cases}
\end{align}
which satisfies
\begin{align}\label{uvpaspatialboundariesacc}
\uvp_a(0)=a \quad \text{and} \quad \uvp_a(s_0)=\f{m}{\omega_n}-\eps
\end{align}
as well as
\begin{align}\label{uvpaprime12geq0acc}
\uvp_a'(s) \geq 0 \quad \text{and} \quad \uvp_a''(s) \geq 0 \qquad \text{for all }s \in [0,s_0].
\end{align}
Now we recursively define a sequence $(a_k)_{k\in\N_0}$ with $a_0$ as in \eqref{a0defindiraccompleteacc} and
\begin{align}\label{akdefacccomplete}
a_{k+1}:=\uvp_{a_k}\Big(\f{s_0}2\Big)=a_k+\f1{27}\Big(\f{m}{\omega_n}-\eps-a_k\Big),
\end{align}
which can easily be seen to exhibit the asymptotic behavior $a_k \nearrow \f{m}{\omega_n}-\eps$ as $k\to\infty$.\\
In turn, writing $t_k:=t_0+k\tau_0$ and $\psi_k:=\psi_{t_k}$ for $k\in \N_0$, we establish
\begin{align*}
\uw_k(s,t):= \uvp_{a_k}(\psi_k (s,t)), \qquad (s,t)\in [0,s_0]\times[t_k,t_{k+1}),
\end{align*}
and observe that accordingly, $\uw_k$ lies in $C^1([0,s_0]\times[t_k,t_{k+1}))\cap C^2(\mathcal G_{t_k})$ and, as $\psi_{ks}\geq 0$ by \eqref{psisislocallyboundedacc}, satisfies
\begin{align}\label{psidiffineqfinal2acc}
\uw_{kt} - n^2s^\alpha \uw_{kss} - n \uw_k \uw_{ks} + \mu s \uw_{ks} &= \uvp_{a_k}' \cdot \psi_{kt} - n^2 s^\alpha(\uvp_{a_k}'' \cdot \psi_{ks}^2 + \uvp_{a_k}' \cdot \psi_{kss}) - n \uvp_{a_k}\cdot\uvp_{a_k}'\cdot \psi_{ks} + \mu s \uvp_{a_k}' \cdot \psi_{ks} \notag\\
&\leq \uvp_{a_k}' \cdot \Big\{ \psi_{kt} - n^2 s^\alpha \psi_{kss} - n \uvp_{a_k} \psi_{ks} + \mu s\psi_{ks} \Big\} \notag\\
&\leq \uvp_{a_k}' \cdot \Big\{ \psi_{kt} - n^2 s^\alpha \psi_{kss} - n a_k \psi_{ks} + \mu s\psi_{ks} \Big\} \notag\\
&\leq \uvp_{a_k}' \cdot \Big\{ \psi_t - n^2 s^\alpha \psi_{kss} - (n a_0 - \mu s)\psi_{ks} \Big\} \notag\\
&\leq 0
\end{align}
in $\mathcal G_{t_k}$ on account of \eqref{uvpaprime12geq0acc}, \eqref{uvpaspatialboundariesacc} and \eqref{psidiffineqacc}.\\
Once more utilizing \eqref{psisislocallyboundedacc}, $\uw_{ks} = \uvp_{a_k}' \cdot \psi_{ks}$ is bounded locally in $[0,s_0]\times[t_k,t_{k+1})$ for each $k\in\N_0$.\\
By \eqref{psi0at0leqs0ats0acc} and \eqref{uvpaspatialboundariesacc}, we verify
\begin{align}\label{uwkisakat0acc}
\uw_k(0,t)=\uvp_{a_k}(\psi_k(0,t))=\uvp_{a_k}(0)=a_k \qquad \text{for all }t\in[t_k,t_{k+1}),
\end{align}
and, additionally resorting to the monotonicity of $\uvp_{a_k}$ and \eqref{ws04greatenoughacc}, we may infer that
\begin{align}\label{uwkissmallerats0acc}
\uw_k(s_0,t) &= \uvp_{a_k}(\psi_k(s_0,t)) \notag\\
&\leq \uvp_{a_k}(s_0) \notag\\
&= \f{m}{\omega_n}-\eps \notag\\
&\leq w\Big(\f{s_0}4,t\Big) \notag\\
&\leq w(s_0,t) \qquad \text{for all }t\in[t_k,t_{k+1}).
\end{align}
Now by induction, we assume that $w(s,t_k) \geq a_k$ for all $s\in(0,s_0]$; for $k=0$, this holds due to \eqref{a0defindiraccompleteacc} and $t_0 \geq \bar t$.\\
Then, again in light of the monotonicity of $\uvp_{a_k}$, \eqref{psiissmallersinitiallyacc}, \eqref{uvpadefacccomplete} and \eqref{ws04greatenoughacc}, it follows that
\begin{align}\label{uwissmallerinitially2acc}
\uw_k(s,t_k)&=\uvp_{a_k}(\psi_k(s,t_k)) \notag\\
&\leq \uvp_{a_k}(s) \notag\\
&\leq w(s,t_k) \qquad \text{for all }s\in(0,s_0].
\end{align}
Given \eqref{psidiffineqfinal2acc}, \eqref{uwkisakat0acc}, \eqref{uwkissmallerats0acc}, \eqref{uwissmallerinitially2acc}, the fact that $\uw_{ks}$ is bounded locally in time, and the discreteness of the set $N(t)$ for each $t\in[t_k,t_{k+1})$ in the definition of $\mathcal G_{t_k}$, we may invoke the comparison principle Lemma \ref{halfcomparison} to deduce that
\begin{align*}
\uw_k(s,t) \leq w(s,t) \qquad \text{for all }s\in(0,s_0]~\text{and}~t\in[t_k,t_{k+1}).
\end{align*}
In conjunction with the continuity of $w(s,\cdot)$ for each $s\in(0,s_0]$ and \eqref{limpsigeqs02acc}, we may deduce that
\begin{align*}
w(s,t_{k+1}) &= \lim_{t \nearrow t_{k+1}} w(s,t) \\
&\geq \lim_{t \nearrow t_{k+1}} \uw_k(s,t) \\
&= \lim_{t \nearrow t_{k+1}} \uvp_{a_k}(\psi_k(s,t)) \\
&= \uvp_{a_k}\Big(\lim_{t \nearrow t_{k+1}}\psi_k(s,t)\Big) \\
&\geq \uvp_{a_k}\Big(\f{s_0}2\Big) \qquad \text{for all }s\in(0,s_0].
\end{align*}
As $\uvp_{a_k}(\f{s_0}2)=a_{k+1}$ by \eqref{akdefacccomplete}, this results in
\begin{align*}
w(s,t_{k+1}) \geq a_{k+1} \qquad \text{for all }s\in(0,s_0].
\end{align*}
Therefore, together with the monotonicity of $t \mapsto \inf_{s\in(0,R^n)} w(s,t)$ asserted by Lemma \ref{infmoninclemacc}, we may infer that
\begin{align*}
\lim_{t \to \infty} \inf_{s\in(0,R^n)} w(s,t) &= \lim_{k \to \infty} \inf_{s\in(0,R^n)} w(s,t_k) \\
&\geq \lim_{k \to \infty} a_k \\
&= \f{m}{\omega_n}-\eps.
\end{align*}
Since $\eps>0$ was arbitrary and $w(s,t) \leq \f{m}{\omega_n}$ for all $(s,t)\in(0,R^n]\times[0,\infty)$, this readily entails \eqref{massasymptoticallyaggregatedacc}.
\end{bew}

\section{Proof of main results}

Combining the preceding lemmata yields the propositions from Subsection \ref{subpropos}.

\begin{proof}[\textbf{Proof of Proposition~\ref{prop1acc}}:]
Lemma \ref{createdsol} guarantees the existence of a nonnegative classical solution $w\in C^0((0,R^n]\times [0,\infty)) \cap C^{2,1}((0,R^n]\times(0,\infty))$ of \eqref{0wacc}, whereas Lemma \ref{minimalsolacc} asserts that this solution is minimal in the sense that
\[
\tilde w \geq w \qquad \text{in }(0,R^n]\times[0,\infty)
\]
for all nonnegative classical solutions $\tilde w$ to \eqref{0wacc}, and that $w_s>0$ in $(0,R^n)\times(0,\infty)$. Furthermore, Lemma \ref{rightcontinuityacc} and Lemma \ref{infmoninclemacc} establish the right-continuity and increasing monotonicity of $t \mapsto w(0,t)=\inf_{s\in(0,R^n)}w(s,t)$, respectively.\\
If additionally
\begin{align*}
\uc \cdot \f{s^2}{s^{2-\gamma_1}+\delta} \leq w_0(s) \leq \oc \cdot s^{\gamma_2} \qquad \text{for all }s\in[0,R^n],
\end{align*}
for $\delta>0$ as in Corollary \ref{corslowcollapseacc}, the corollary yields
\begin{align*}
0<w(0,t_0)<\eta \qquad \text{for some }t_0\in (0,\tau).
\end{align*}
\end{proof}

\begin{proof}[\textbf{Proof of Proposition~\ref{prop2acc}}:]
The strict monotonicity is a direct consequence of Lemma \ref{strictmonotonacclemma}, while the convergence property has been established in Lemma \ref{massaccumulatestotallyacc}.
\end{proof}

The final task in proving Theorem \ref{maintheoremacc} is to construct a measure-valued solution of \eqref{KS} from $w$; all asserted properties of this solution then follow immediately from the preceding propositions.

\begin{proof}[\textbf{Proof of Theorem~\ref{maintheoremacc}}:]
Let $w\in C^0((0,R^n]\times [0,\infty)) \cap C^{2,1}((0,R^n]\times(0,\infty))$ be a classical nonnegative solution of \eqref{0wacc} with $w_0$ as in \eqref{w0acc} satisfying $w_s\geq0$ in $(0,R^n]\times(0,\infty)$.\\
For $t\geq 0$ and $r:=|x|$, define $u(t)\in \mathcal M_+(\bom)$ as a radially symmetric measure via
\begin{align}\label{udefmeasureacc}
\begin{cases}
u(t)(B_r(0)):= \omega_n w(r^n,t), \qquad r\in(0,R],\\
u(t)(\{0\}):= \omega_n w(0,t),
\end{cases}
\end{align}
and let $v\in L_{loc}^1([0,\infty);W^{1,1}(\Om))\cap C^{1,0}((\bom\setminus\{0\})\times(0,\infty))$ be such that $\io v(\cdot,t)=0$ for all $t>0$ and
\[
\nabla v(x,t) = r^{1-n}(w(r^n,t)- \f{\mu}{n}r^n) \cdot \f{x}{r}, \qquad x\in\bom\setminus\{0\},~t>0.
\]
Since $n \geq 2$, the set $\bom\setminus\{0\}$ is connected and therefore $v$ is uniquely determined even without assuming radial symmetry a priori.\\
For any radially symmetric function $\psi \in C^\infty(\bom)$, with $\psi(x)=\Psi(r)$ we may calculate $\nabla \psi(x)= \Psi^\prime (r) \cdot \f{x}{r}$ and therefore via integration by parts receive
\begin{align}\label{vpsitestacc}
\int_{\bom} \nabla v \cdot \nabla \psi dx &= \omega_n \int_0^R v_r(r,t) \cdot \Psi^\prime(r) \cdot r^{n-1} dr \notag\\
&=\omega_n \int_0^R r^{1-n}(w(r^n,t)- \f{\mu}{n}r^n) \cdot \Psi^\prime(r) \cdot r^{n-1} dr \notag\\
&=\omega_n \int_0^R (w(r^n,t)- \f{\mu}{n}r^n) \cdot \Psi^\prime(r) dr \notag\\
&=\omega_n\cdot\Big[(w(R^n,t)-\f{\mu}{n}R^n)\Psi(R)-w(0,t)\cdot \Psi(0)\Big]-\omega_n \int_0^R (w(r^n,t)-\f{\mu}{n}r^n)\Psi(r)dr \notag\\
&= -\omega_n w(0,t) \cdot \Psi(0) - \omega_n \int_0^R (nr^{n-1} w_s(r^n,t) - \mu r^{n-1})\Psi(r)dr
\end{align}
as $w(R^n,t)=\f{m}{\omega_n}=\f{\mu}{n}R^n$, whereas on the other hand
\begin{align}\label{iomupsitestacc}
\int_{\bom} \mu \psi dx = \mu \omega_n \int_0^R \Psi(r)\cdot r^{n-1} dr
\end{align}
and
\begin{align}\label{ioupsitestacc}
\int_{\bom} \psi du(t) = \omega_n w(0,t) \Psi(0) + \omega_n \int_0^R \Psi(r) \cdot nr^{n-1} w_s(r^n,t) dr,
\end{align}
whence combining \eqref{vpsitestacc}, \eqref{iomupsitestacc} and \eqref{ioupsitestacc} implies that
\[
\int_{\bom} \nabla v \cdot \nabla \psi dx - \int_{\bom} \mu \psi dx + \int_{\bom} \psi du(t) = 0,
\]
confirming \eqref{vdistreqacc}. 
The representation 
\begin{align}\label{upartingproofacc}
u(t)=\theta(t)\delta_0 + \rho(\cdot,t) dx, \qquad t\geq 0,
\end{align}
with 
\[
\theta(t)=\omega_n w(0,t) \quad \text{for }t\geq 0 \quad \text{and} \quad \rho(x,t)= nw_s(r^n,t) \quad \text{for }(x,t)\in\bom\times[0,\infty)
\]
follows directly from \eqref{udefmeasureacc} and \eqref{0wacc}. The asserted regularity classes $\theta \in L^\infty((0,\infty))$ and $\rho \in C^{1,1}((\bom\setminus\{0\})\times[0,\infty))$ are thus a consequence of the boundedness and regularity of $w$, respectively.\\
Hence we may conclude that $(u,v)$ defines a mass-conserving radially symmetric measure-valued solution of \eqref{KS}, and thus by definition, mass-conserving radially symmetric measure-valued solutions of \eqref{KS} correspond to classical solutions of \eqref{0wacc}. In particular, minimality of $w$ as for the solution in Proposition \ref{prop1acc} directly translates to minimality of $u$. Notably, $\rho \in C^0((\bom\setminus\{0\})\times[0,\infty))$ for the minimal solution derives from the classical solvability of \eqref{KS} in $\Om\times(0,T_{max})$ for some $T_{max}>0$.\\
The nonnegativity of $\theta$ and $\rho$ follow directly from that of $w$ and $w_s$, respectively, whereas the monotonicity and right-continuity of $\theta=w(0,\cdot)$ have been established in Proposition \ref{prop1acc}.\\
For the next part, we first demonstrate that for $\gamma\in(0,1-\f2n)$, $w_0$ as in \eqref{w0acc} satisfies
\begin{align*}
w_0(s) &= \int_0^{s^\frac{1}{n}} \rho^{n-1} u_0(\sigma) d\sigma \\
&\leq \f1n \|u_0\|_{L^\infty(\Om)}\cdot s \\
&\leq \f1n \|u_0\|_{L^\infty(\Om)}R^{n(1-\gamma)}\cdot s^\gamma, \qquad s\in [0,R^n],
\end{align*}
that is $w_0(s)\leq \oc s^\gamma$ with $\oc := \f1n \|u_0\|_{L^\infty(\Om)}R^{n(1-\gamma)}$ and thus the inequality on the right-hand side of \eqref{prop1accw0bounds}. Further given $c>0$, $\tau>0$ and $\eta>0$, we may once more invoke Proposition \ref{prop1acc} with $\eta'=\f\eta{\omega_n}$ to assert the existence of $\delta>0$ such that if
\[
w_0(s) \geq c \cdot \f{s^2}{s^{2-\gamma}+\delta} \qquad \text{for all }s\in[0,R^n],
\]
then for some $t_0\in(0,\tau)$,
\[
0<\theta(t_0)<\eta.
\]
Finally, the strict monotonicity of $\theta$ in $[t_0,\infty)$ and the convergence property are immediate consequences of Proposition \ref{prop2acc}.
\end{proof}

\section*{Acknowledgements}
The author acknowledges support of the {\em Deutsche Forschungsgemeinschaft} in the context of the project
  {\em Fine structures in interpolation inequalities and application to parabolic problems}, project number 462888149.

\section*{Disclosure of interest}
The author reports there are no competing interests to declare.

\footnotesize{
\setlength{\bibsep}{3pt plus 0.5ex}
\bibliography{mass_acc}
\bibliographystyle{abbrvurl}
}

\end{document}